\def\BibTeX{{\rm B\kern-.05em{\sc i\kern-.025em b}\kern-.08em
		T\kern-.1667em\lower.7ex\hbox{E}\kern-.125emX}}
\documentclass[letterpaper,journal,final]{IEEEtran}
\usepackage{amsthm}
\usepackage{xcolor}

\usepackage{xparse}
\usepackage{bm}
\usepackage{cite}
\usepackage{amsmath,amssymb,amsfonts}
\usepackage{mathrsfs}
\usepackage{algorithmic}
\usepackage{graphicx}
\usepackage{algorithm,algorithmic}
\usepackage{hyperref}
\hypersetup{hidelinks=true}
\usepackage{textcomp}
\usepackage{stmaryrd}
\usepackage{siunitx}
\usepackage{booktabs}

\usepackage{enumerate}
\usepackage{aligned-overset}

\usepackage{tikz}
\usetikzlibrary{decorations.markings,arrows.meta,decorations.pathreplacing,arrows,shapes}
\usetikzlibrary{patterns}
\usetikzlibrary{automata,positioning}
\usepackage{pgfplots}
\usepgfplotslibrary{groupplots}
\usepackage{multirow}

\newcommand{\naturals}{\mathbb{N}}
\newcommand{\reals}{\mathbb{R}}

\newcommand{\X}{\mathcal{X}}
\newcommand{\U}{\mathcal{U}}
\newcommand{\W}{\mathcal{W}}
\newcommand{\T}{\mathcal{T}}
\newcommand{\tend}{t_{\mathrm{end}}}
\newcommand{\R}{\mathcal{R}}

\newcommand{\ball}{\mathcal{B}}

\DeclareMathOperator*{\argmax}{argmax}
\newcommand{\trace}{\mathrm{tr}}
\newcommand{\conv}{\mathrm{conv}}

\DeclareMathOperator{\Diag}{Diag}

\DeclareMathOperator{\rank}{rank}

\renewcommand{\v}[1]{\vec{#1}}
\newcommand{\mat}[1]{\bm{\underline{\smash{#1}}}}

\newcommand{\problem}[1]{\mathsf{#1}}
\newcommand{\algo}[1]{\mathtt{#1}}
\newcommand{\Poly}{\problem{P}}
\newcommand{\NP}{\problem{NP}}

\newcommand{\RQP}{\problem{RQP}}
\newcommand{\PTAS}{\problem{PTAS}}
\newcommand{\FPTAS}{\problem{FPTAS}}
\newcommand{\coNP}{\problem{co}\text{-}\problem{NP}}

\newcommand{\RP}{\problem{RP}}

\newcommand{\Set}[2]{\left\{#1\;\middle|\;#2\right\}}
\NewDocumentCommand{\iprod}{gg}{\left\langle\IfValueTF{#1}{#1}{\cdot},\IfValueTF{#2}{#2}{\cdot}\right\rangle}
\NewDocumentCommand{\absiprod}{gg}{\left|\left\langle\IfValueTF{#1}{#1}{\cdot},\IfValueTF{#2}{#2}{\cdot}\right\rangle\right|}
\NewDocumentCommand{\norm}{g}{\left\|\IfValueTF{#1}{#1}{\cdot}\right\|}
\NewDocumentCommand{\nnorm}{g}{\|\IfValueTF{#1}{#1}{\cdot}\|}

\newcommand{\Nts}{N_{\mathrm{ts}}}

\makeatletter
\DeclareRobustCommand\widecheckinternal[1]{{\mathpalette\@widecheckinternal{#1}}}
\def\@widecheckinternal#1#2{%
	\setbox\z@\hbox{\m@th$#1#2$}%
	\setbox\tw@\hbox{\m@th$#1%
		\widehat{\vrule\@width\z@\@height\ht\z@
			\vrule\@height\z@\@width\wd\z@}$}%
	\dp\tw@-2\ht\z@
	\@tempdima\ht\z@ \advance\@tempdima2\ht\tw@ \divide\@tempdima\thr@@
	\setbox\tw@\hbox{\raise1.05\@tempdima\hbox{\scalebox{1}[-1]{\lower\@tempdima\box\tw@}}}%
	{\ooalign{\box\tw@ \cr \box\z@}}}
\makeatother
\newcommand{\widecheck}[1]{\,\widecheckinternal{\kern -2pt #1}}

\newcommand{\Zcirc}{\widehat{Z}}
\newcommand{\Zin}{\widecheck{Z}}
\newcommand{\Ecirc}{\widehat{E}}
\newcommand{\Ein}{\widecheck{E}}

\newcommand{\ST}{\mathtt{LR}}
\newcommand{\SR}{\mathtt{SR}}

\newcommand{\ZSR}{\mathtt{ZSR}}

\newtheorem{theorem}{Theorem}
\newtheorem{proposition}[theorem]{Proposition}%
\newtheorem{lemma}[theorem]{Lemma}%

\newtheorem{example}[theorem]{Example}%
\newtheorem{remark}[theorem]{Remark}%

\newtheorem{definition}[theorem]{Definition}%

\definecolor{Blue}{RGB}{0,92,171}
\definecolor{Red}{RGB}{227,27,35}
\definecolor{Yellow}{RGB}{255,195,37}
\definecolor{YellowDark}{RGB}{219, 126, 8}

\newcommand{\padding}{0.3}

\pgfdeclarepatternformonly{my large dots}{\pgfqpoint{-1pt}{-1pt}}{\pgfqpoint{5pt}{5pt}}{\pgfqpoint{6pt}{6pt}}%
{
	\pgfpathcircle{\pgfqpoint{3pt}{3pt}}{2pt}
	\pgfusepath{fill}
}

\begin{document}
	\title{Approximability of the Containment Problem for Zonotopes and Ellipsotopes}
	\author{Adrian Kulmburg, Lukas Schäfer, and Matthias Althoff
		\thanks{This work has been accepted for publication in the \emph{IEEE Transactions on Automatic Control}. The final version is available at \url{https://doi.org/10.1109/TAC.2025.3583624}
			\newline
	        The authors gratefully acknowledge financial support by the project justITSELF funded by the European Research Council (ERC) under grant agreement No 817629, by the German Research Foundation (DFG) under grant numbers AL 1185/19-1 and SFB 1608 (Convide), and the German Federal Ministry for Economics Affairs and Climate Action project VaF under grant KK5135901KG0.}
		\thanks{The authors are with the School of Computation, Information and Technology, Technical University of Munich, Boltzmannstr. 3, 85748 Garching b. München, Germany.\\
		Email: \texttt{\{adrian.kulmburg,lukas.schaefer,althoff\}@tum.de}
		}
	}
	
	\maketitle
	
	\begin{abstract}
		The zonotope containment problem, i.e., whether one zonotope is contained in another, is a central problem in control theory. Applications include detecting faults and robustifying controllers by computing invariant sets, and obtaining fixed points in reachability analysis. Despite the inherent $\coNP$-hardness of this problem, an approximation algorithm developed by S. Sadraddini and R. Tedrake has gained widespread recognition for its swift execution and consistent reliability in practice. In our study, we substantiate the precision of the algorithm with a definitive proof, elucidating the empirical accuracy observed in practice. Our proof hinges on establishing a connection between the containment problem and the computation of matrix norms, thereby enabling the extension of the approximation algorithm to encompass ellipsotopes -- a broader class of sets derived from zonotopes. We also explore the computational complexity of the ellipsotope containment problem with a focus on approximability. Finally, we present new methods to compute safe sets for linear dynamical systems, demonstrating the practical relevance of approximating the ellipsotope containment problem.
	\end{abstract}
	

	\section{Introduction}
	The containment problem asks whether a set $\widecheck{S}$ (the \emph{inbody}) is contained in a set $\widehat{S}$ (the \emph{circumbody}). This problem appears naturally in the context of reachability analysis \cite{althoff_reachability_2013}, robust control \cite{rakovic_optimized_2007}, controller synthesis \cite{ren_zonotope-based_2021}, conformance checking \cite{Roehm2016}, but has also been a subject of study of its own for spectrahedra \cite{kellner_containment_2013} and polytopes \cite{gritzmann_complexity_1994}. For certain applications, it is not enough to verify whether $\widecheck{S}$ is contained in $\widehat{S}$ -- one may rather be interested in estimating \emph{how far away} $\widecheck{S}$ is from being contained in $\widehat{S}$. One way to quantify this, if $\widehat{S}$ has a center $\v{c}_{\widehat{S}}$, is by searching for the smallest scaling factor $r(\widecheck{S},\widehat{S})\geq 0$ such that
	\begin{equation}
		\label{eq:intro_containment}
		\widecheck{S} \subseteq r(\widecheck{S},\widehat{S}) \cdot (\widehat{S}-\v{c}_{\widehat{S}})+\v{c}_{\widehat{S}}.
	\end{equation}
	If $\widehat{S}\subset \reals^n$ is compact, convex, and centrally symmetric around $\v{c}_{\widehat{S}}$, the set $\widehat{S}$ induces a norm $\norm_{\widehat{S}}$ on $\reals^n$ with unit ball $\Set{\v{x}\in\reals^n}{\norm{\v{x}}_{\widehat{S}} \leq 1}$ coinciding $\widehat{S}-\v{c}_{\widehat{S}}$, and
	\begin{equation}
		\label{eq:r_def_general}
		r(\widecheck{S},\widehat{S}) = \max_{\v{s} \in \widecheck{S}} \norm{\v{s} - \v{c}_{\widehat{S}}}_{\widehat{S}}.
	\end{equation}
	In many cases, calculating $r(\widecheck{S},\widehat{S})$ (or, equivalently, finding the minimal $r(\widecheck{S}, \widehat{S})$ such that \eqref{eq:intro_containment} holds) is known to be hard, as demonstrated in \cite{gritzmann_complexity_1994} for polytopes, in \cite{kellner_containment_2013} for spectrahedra, or in \cite{Kulmburg2021} for zonotopes. On the other hand, \cite{sadraddini_linear_2019} describes a method that gives a sufficient (but not necessary) condition to check containment for two polytopes and, in particular, for two zonotopes. This method is fast (it only requires solving a linear program) and surprisingly accurate in practice for the zonotope containment. We will discuss at the beginning of Section \ref{sec:containment_problem} how this sufficient criterion for zonotope containment is equivalent to an approximation of \eqref{eq:r_def_general}, providing us with a way to quantify the accuracy of the algorithm. We will also establish a connection between $r$ and certain matrix norms, allowing us to generalize the criterion from \cite{sadraddini_linear_2019} to the ellipsotope containment problem. Ellipsotopes \cite{ellipsotopes} are a generalization of zonotopes that can also represent ellipsoids and centrally symmetric polytopes.
	
	The core of Section \ref{sec:containment_problem} consists of a case-by-case presentation of approximations to compute $r$. We continue by investigating the computational complexity of the ellipsotope containment problem in Section \ref{sec:hardness}, showing that most instances are hard to approximate (unless $\Poly = \NP$). Finally, in Section \ref{sec:applications} we present a novel procedure, based on \cite{gruber}, to compute safe sets of linear dynamical systems using approximations to the containment problem for zonotopes and ellipsoids.
	
	\section{Preliminaries}
	\subsection{Basic Notation}
	Any letter with an arrow $\v{v}$ represents a \emph{vector} in $\reals^n$ or a vector field $\reals^m \rightarrow \reals^n$, \emph{matrices} in $\reals^{n \times m}$ are denoted by bold, underlined letters (e.g., $\mat{M}$).
	The vectors $\v{e}_i \in \reals^n$ for $i=1,...,n$ are the \emph{canonical basis vectors} of $\reals^n$. For a vector $\v{v}\in \reals^n$, $v_i$ is the $i$-th coordinate of $\v{v}$. Similarly, $M_{ij}$ is the entry in the $i$-th row and $j$-th column of $\mat{M}\in\reals^{n\times m}$.
	The matrix $\mat{I}_n$ is the $n$-dimensional \emph{identity matrix}, $\mat{0}_{n\times m}$ the $\bm{n\times m}$\emph{-matrix filled with zeros}, $\v{1}_n$ the $n$-dimensional \emph{vector filled with ones}, and $\v{0}_n$ the $n$-dimensional \emph{vector filled with zeros}. If the dimension is clear from the context, we ommit the index. For $\v{v}\in\reals^n$, $\Diag(\v{v})$ is the \emph{diagonal matrix} with $\v{v}$ as its diagonal. For a matrix $\mat{M}$, $\mat{M}^+$ is its \emph{Moore-Penrose pseudo-inverse}, $\trace(\mat{M})$ is its \emph{trace}, $\text{Im}(\mat{M})$ is its image, $\rank(\mat{M})$ is its rank, and $|\mat{M}|$ is the element-wise \emph{matrix of absolute values} of $\mat{M}$. For matrices $\mat{A}, \mat{B} \in \reals^{n\times n}$, $\mat{A} \succeq \mat{B}$ means that $\mat{A}-\mat{B}$ is positive semidefinite, and $\mat{A} \succeq 0$ means $\mat{A}$ is positive semidefinite. For vectors $\v{v}_1,...,\v{v}_N\in\reals^n$, $\conv(\v{v}_1,...,\v{v}_N)$ is their convex hull.
	If $\mat{A}, \mat{B}$ are matrices (or vectors), $[\mat{A} \; \mat{B}]$ is the \emph{horizontal} and $\begin{bmatrix}\mat{A}\\\mat{B}\end{bmatrix}$ the \emph{vertical concatenation} of $\mat{A}$ and $\mat{B}$ (assuming the dimensions of $\mat{A}$ and $\mat{B}$ match).
	For $S, T \subseteq\reals^n$, $S \oplus T := \Set{s+t}{s\in S, t\in T}$ is the \emph{Minkowski sum} of $S$ and $T$.
	For $\v{v},\v{w}\in\reals^n$, their (Euclidean) \emph{inner product} is $\v{v}^\top\v{w} = v_1w_1+...+v_nw_n$. For $p\in[1,\infty)$, $\norm{\v{v}}_p := \sqrt[p]{|v_1|^p+...+|v_n|^p}$ is the $\bm{p}$\emph{-norm} of $\v{v}$, for $p=\infty$, $\norm{\v{v}}_{\infty} := \max_i |v_i|$. We denote the \emph{unit ball} of the $p$-norm as $\ball_p$, and $p^*$ the \emph{Hölder conjugate} of $p$, defined through $\frac{1}{p}+\frac{1}{p^*} = 1$.
	\subsection{Matrix Norms}
	As mentioned in the introduction, solving a containment problem often involves the computation of a norm. We begin by defining the $L_{p,q}$-norms and $L_{p,q}^\top$-norms:
	\begin{definition}[$L_{p,q}$-norms and $L_{p,q}^\top$-norms]
		\label{def:matrix_L_p_q_norms}
		For $p,q\in[1,\infty]$ and a matrix $\mat{A}\in\reals^{n\times m}$ with columns $\v{a}_1,\hdots,\v{a}_m$, the $L_{p,q}$-norm is
		\begin{equation}
			\norm{\mat{A}}_{L_{p,q}} = \norm{ \begin{pmatrix}\norm{\v{a}_1}_p & \cdots & \norm{\v{a}_m}_p\end{pmatrix}^\top }_q.
		\end{equation}
		For $p, q \in [1, \infty)$, this can also be written as
		\begin{equation}
			\label{eq:L_pq_def}
			\norm{\mat{A}}_{L_{p,q}} = \left(\sum_{j=1}^m \left( \sum_{i=1}^n |A_{ij}|^p\right)^{q/p}\right)^{1/q}.
		\end{equation}
		The transposed $L_{p,q}$-norms, or $L_{p,q}^\top$-norms, are defined as
		\begin{equation}
			\norm{\mat{A}}_{L_{p,q}^\top} = \nnorm{\mat{A}^\top}_{L_{p,q}}.
		\end{equation}
	\end{definition}
	
	One can also define norms on matrices through the operator norm (see \cite[Chapter 5.6]{horn2012}):
	\begin{definition}[Operator Matrix $p\mapsto q$-norms]
		\label{def:matrix_p_q_norms}
		For a matrix $\mat{A}\in\reals^{n\times m}$ and numbers $p,q\in[1,\infty]$, the operator $p\mapsto q$-norm of $\mat{A}$ is
		\begin{equation}
			\norm{\mat{A}}_{p\mapsto q} = \sup_{\norm{\v{v}}_p\leq1} \norm{\mat{A}\v{v}}_q.
		\end{equation}
		If $p=q$, we will use $\norm{\mat{A}}_p$ as a shorthand for $\norm{\mat{A}}_{p \mapsto p}$.
	\end{definition}
	
	For $p=q=1$ and $p=q=\infty$, the $p\mapsto q$-norm coincides with an entry-wise norm:
	\begin{equation}
		\label{eq:11_L1oo}
		\norm{\mat{A}}_{1} = \norm{\mat{A}}_{L_{1,\infty}}, \quad \text{ and } \quad \norm{\mat{A}}_{\infty} = \norm{\mat{A}}_{L_{1,\infty}^\top}.
	\end{equation}
	
	For a norm $\norm$ on $\reals^n$, $\norm^*$ is its \emph{dual norm} (see \cite[Chapter 1.3]{buehler2018}), defined as
	\begin{equation}
		\label{eq:dual_norm_def}
		\norm{\v{y}}^* = \sup_{\norm{\v{x}} \leq 1} \v{y}^{\top}\vec{x}.
	\end{equation}
	As pointed out in \cite[Chapter A.1.6]{boyd_convex_2004}, for the vector $p$-norms,
	\begin{equation}
		\|\v{x}\|_{p}^* = \|\v{x}\|_{p^*}.
	\end{equation}
	For a norm $\norm$ on the space of matrices $\reals^{n\times m}$, the dual norm (see \cite[Chapter A.1.6]{boyd_convex_2004}) is
	\begin{equation}
		\norm{\mat{Y}}^* = \sup_{\norm{\mat{X}}\leq 1} \trace(\mat{X}^\top \mat{Y}).
	\end{equation}
	
	\subsection{Polyhedra and Polytopes}
	Polyhedra and polytopes are perhaps the most common convex set representations and can be defined as follows:
	\begin{definition}[Polyhedra and Polytopes]
		For $k,n\in\naturals$, $\mat{\Lambda} \in \reals^{k\times n}$, and $\v{\lambda} \in \reals^k$, the set
		\begin{equation}
			\label{eq:polyhedron_def}
			P(\mat{\Lambda}, \v{\lambda}) = \Set{\v{x}\in\reals^n}{\mat{\Lambda}\v{x} \leq \v{\lambda}}
		\end{equation}
		is a (convex) \emph{polyhedron} (or \emph{H-polyhedron}). The inequality in \eqref{eq:polyhedron_def} is to be understood component-wise, i.e., $(\mat{\Lambda}\v{x})_i \leq \lambda_i$ for all $i=1,...,k$, where $\mat{\Lambda}$ is the \emph{halfspace matrix} and $\v{\lambda}$ the \emph{halfspace offset}. A bounded polyhedron is called a \emph{polytope} (or \emph{H-polytope}), in which case it can also be represented as the convex hull $\conv(\v{v}_1,...,\v{v}_N)$ of the points $\v{v}_i\in \reals^n$ with $i=1,...,N$ (see \cite[Theorem 1.1]{ziegler2012lectures}). If a polytope is represented this way, it is also referred to as a \emph{V-polytope}. For a polytope or polyhedron $P \subseteq \reals^n$, we say that $P$ is \emph{symmetric} if $\v{x} \in P \Leftrightarrow -\v{x} \in P$ for all $\v{x}\in\reals^n$.
	\end{definition}
	\begin{remark}
		Note that a symmetric V-polytope $P\subset \reals^n$ can be written as $\Set{\mat{V}\v{\alpha}}{\norm{\v{\alpha}}_1 \leq 1}$, where the columns of $\mat{V} \in \reals^{n\times m}$ are the vertices of $P$.
	\end{remark}

	\subsection{Ellipsotopes, Zonotopes, and Ellipsoids}
	Let us now introduce our primary object of study, namely ellipsotopes (see \cite[Definition 2]{ellipsotopes}):
	\begin{definition}[Ellipsotopes]
		\label{def:ellipsotopes}
		For $m,n \in \naturals$, $\mat{G} \in \reals^{n\times m}$, $\v{c} \in \reals^n$, and $p\in[1,\infty]$, the set
		\begin{equation}
			E_p(\mat{G},\v{c}) = \Set{\v{c} + \mat{G}\v{\alpha}}{\norm{\v{\alpha}}_{p}\leq 1}
		\end{equation}
		is a basic \emph{ellipsotope} (or $p$-\emph{ellipsotope}) with \emph{center} $\v{c}$ and \emph{generator matrix} $\mat{G}$, with column vectors $\v{g}_1,...,\v{g}_m$ called \emph{generators}. An ellipsotope is \emph{non-degenerate}, if $\mat{G}$ has full rank and $n\leq m$.	
	\end{definition}
	\begin{remark}
		In the present article, for simplicity the term \emph{ellipsotope} will refer to \emph{basic ellipsotopes} as defined above in Definition \ref{def:ellipsotopes}. Please note that the class of non-basic ellipsotopes as defined in \cite{ellipsotopes} encompasses a broader range of sets.
	\end{remark}
	We will use the shorthand $E_p(\mat{G})$ for $E_p(\mat{G}, \v{0})$.
	Ellipsotopes with $p=2$ and $p=\infty$ are called \emph{ellipsoids}\footnote{Ellipsoids can also be represented as the solution set of convex quadratic forms \cite[Definition 2]{victor_ellipsoid} or through support functions \cite[Definition 2.1.3]{ellipsoidal_toolbox}. Due to space restrictions, we leave it to the reader to verify that one can change from one representation to another in polynomial time with respect to the size of the generator matrix or the dimension.} and \emph{zonotopes}, respectively, and we use the shorthand $E(\mat{G},\v{c})$, $E(\mat{G})$, $Z(\mat{G},\v{c})$, and $Z(\mat{G})$ for $E_2(\mat{G},\v{c})$, $E_2(\mat{G})$, $E_{\infty}(\mat{G},\v{c})$, and $E_{\infty}(\mat{G})$, respectively.
	
	Intuitively, a non-degenerate ellipsotope is a set that has a topological interior, as opposed to a degenerate ellipsotope that is entirely contained in a strict affine subspace of $\reals^n$, and for which every point lies on the (topological) boundary of the ellipsotope. Since non-degenerate ellipsotopes are convex, compact, and centrally symmetric, they induce a norm \cite[Theorem 5.3.1]{narici2010topological}:
	\begin{proposition}[Ellipsotope norm]
		\label{prop:ellipsotope_norm}
		Let $m\geq n$, and suppose $\mat{G} \in \reals^{n\times m}$ has full rank. For any $p\in[1,\infty]$ the function
		\begin{equation}
			\begin{split}
				\v{x} &\mapsto \min_{\mat{G}\v{\alpha} = \v{x}}\norm{\v{\alpha}}_{p}
			\end{split}
		\end{equation}
		defines a norm on $\reals^n$, which we denote by $\norm{\v{x}}_{E_p(\mat{G})}$. The unit ball of this norm coincides with $E_p(\mat{G})$.
	\end{proposition}
	
	\begin{proof}
		The proof is nearly identical to \cite[Section 3.2.]{Kulmburg2021}, as one only needs to replace the $\infty$-norm by the $p$-norm.
	\end{proof}
	
	Note that evaluating $\norm{\v{x}}_{E_p(\mat{G})}$ involves solving a convex minimization problem, which can be done in polynomial time (with respect to the size of $\mat{G}$) up to arbitrary accuracy as described in \cite[Chapter 11]{boyd_convex_2004}.
	
	\subsection{Approximability of Optimization Problems}
	\label{sec:approximability}
	To determine the computational complexity of an optimization problem, one can consider its approximability \cite[p.~68]{vazirani2013approximation}:
	\begin{definition}[Approximation Algorithms]
		We consider a maximization problem $\sup_{y \in \mathcal{D}} f(y)$, where $\mathcal{D}$ and $f$ can be represented using a binary string $\chi$ of length $N$, i.e., $\chi \in \{0,1\}^N$. We refer to $\chi$ as an \emph{instance} or \emph{input} of the problem, $|\chi| := N$ is the \emph{input size}, and $\problem{opt}(\chi) = \sup_{y \in \mathcal{D}} f(y)$ is the \emph{optimal value} (see \cite[Chapter 1.6]{Arora_Barak_2009} and \cite[Chapter 1]{ausiello_complexity_1999}).
		For $\tau\geq 1$, a $\bm{\tau}$\emph{-approximation algorithm} is an algorithm that produces a value $\algo{approx}(\chi)$ such that 
		\begin{equation}
			\label{eq:approx_ratio_max}
			\problem{opt}(\chi) \leq \algo{approx}(\chi) \leq \tau \cdot \problem{opt}(\chi),
		\end{equation}
		and such that the algorithm runs in polynomial time with respect to $|\chi|$. The value $\tau$ is the \emph{approximation ratio} of the algorithm. For minimization problems, the approximation ratio is defined similarly: Instead of \eqref{eq:approx_ratio_max}, $\algo{approx}$ has to satisfy
		\begin{equation*}
			\tau'\cdot \problem{opt}(\chi) \leq \algo{approx}(\chi) \leq \problem{opt}(\chi),
		\end{equation*}
		where $\tau' \leq 1$.
	\end{definition}
	We can then define the class of problems that can be approximated to any desired degree of accuracy, as presented in \cite[p.~68]{vazirani2013approximation}:
	\begin{definition}[$\PTAS$, $\FPTAS$, and Inapproximability]
		The class of optimization problems for which there exists an $(1+\varepsilon)$- or $(1-\varepsilon)$-approximation algorithm for any $\varepsilon>0$ is called $\PTAS$ (\textbf{P}olynomial \textbf{T}ime \textbf{A}pproximation \textbf{S}cheme), whereas $\FPTAS$ (\textbf{F}ully \textbf{P}olynomial \textbf{T}ime \textbf{A}pproximation \textbf{S}cheme) consists of problems from $\PTAS$ that have an $(1+\varepsilon)$- or $(1-\varepsilon)$-approximation algorithm with a polynomial runtime with respect to the input size \emph{and} $1/\varepsilon$.
		
		On the other hand, a maximization\footnote{A similar definition can be made for minimization problems.} problem is called $\bm{\tau}$\emph{-inapproximable} if there cannot exist a $\widetilde{\tau}$-approximation algorithm with $1 \leq \widetilde{\tau} < \tau$, and the problem is \emph{inapproximable} if it is $\tau$-inapproximable $ \forall \tau\geq 1$.
	\end{definition}
	
	\begin{remark}
		\label{rmk:optimization_complexity}
		Most of the convex optimization problems found in the literature (linear, convex quadratic, as well as many instances of semi-definite optimization problems) are in $\FPTAS$, since they can be solved by interior-point methods (see \cite[Chapter 11]{boyd_convex_2004}). Specifically, solving a linear or convex quadratic program in standard form involving $n$ variables and $m$ linear equality constraints with $n \geq m$ requires $\mathcal{O}(n^{3.5}\ln(n))$ operations, assuming a fixed accuracy, as discussed in \cite[Chapter 11]{boyd_convex_2004}. The same is true for nonlinear but convex optimization problems with only linear equality constraints and a nonlinear objective function that can be evaluated in $\mathcal{O}(n)$. Semidefinite programs in standard form involving $n$ variables constrained by an $m\times m$-matrix inequality, with $n \leq m$, require $\mathcal{O}(m^{4.5}\ln(m))$ operations, as shown in \cite[Example~11.9]{boyd_convex_2004}.
	\end{remark}
	
	\begin{remark}
		\label{rmk:higher_inapproximability}
		It is also possible to generalize the concept of inapproximability to cases where $\tau$ is a function of the input size $|\chi|$. For instance, a problem is said to be inapproximable within a factor $\Omega(N^2)$ if there cannot exist an approximation algorithm with $\tau = CN^2$ for any constant $C\geq1$, where $N = |\chi|$.
	\end{remark}

	\section{The Ellipsotope Containment Problem}
	\label{sec:containment_problem}
	We now turn towards the problem of checking whether an ellipsotope $\Ein$ (called the \emph{inbody}) is contained within another ellipsotope $\Ecirc$ (called the \emph{circumbody}). For the remainder of this section we set $\Ein = E_p(\mat{G},\v{c})$ for $\mat{G} \in \reals^{n\times m}, \v{c}\in\reals^n, p\in[1,\infty]$ and $\Ecirc = E_q(\mat{H},\v{d})$ for $\mat{H} \in \reals^{n\times l}, \v{d} \in \reals^n, q\in[1,\infty]$, and we refer to an algorithm as running in polynomial time if it runs in polynomial time with respect to $n, m, l$, and some fixed accuracy for interior point methods as described in \cite[Chapter 11]{boyd_convex_2004}. In \cite[Equation (18)]{Kulmburg2021}, it was shown that the containment problem for zonotopes is equivalent to a certain optimization problem. The corresponding statement for ellipsotopes can be stated as follows:
	\begin{theorem}
		\label{definition:ellipsotope_containment}
		The quantity
		\begin{equation}
			\label{eq:ellipsotope_containment}
			r(\Ein,\Ecirc) = \max_{\norm{\v{\alpha}}_p\leq 1} \min_{\mat{H}\v{\beta} = \mat{G}\v{\alpha} + \v{c} - \v{d}}\nnorm{\v{\beta}}_q
		\end{equation}
		is the smallest scalar such that $\Ein \subseteq r(\Ein,\Ecirc) \cdot(\Ecirc - \v{d}) + \v{d}$, and $r(\Ein,\Ecirc) \leq 1$ if and only if $\Ein \subseteq \Ecirc$. In particular, if $\Ecirc$ is degenerate and
		\begin{equation}
			\label{eq:rank_constraint}
			\rank\left(\begin{bmatrix}\mat{G} & \v{c}-\v{d} & \mat{H}\end{bmatrix}\right) \neq \rank(\mat{H}),
		\end{equation}
		we use the convention $r(\Ein, \Ecirc) = \infty$.
	\end{theorem}
	
	\begin{proof}
		First, let us assume $\Ecirc$ is non-degenerate. For a point $\v{p}\in\reals^n$ and a scalar $\rho \geq 0$, by definition $\v{p} \in \rho \cdot (\Ecirc-\v{d}) + \v{d}$ if and only if there exists a $\v{\beta}'\in\reals^l$ such that $\rho\mat{H}\v{\beta}' + \v{d} = \v{p}$ with $\nnorm{\v{\beta}'}_q \leq 1$, which can occur if and only if there exists a $\v{\beta} \in \reals^l$ such that $\mat{H}\v{\beta}+\v{d} = \v{p}$ with $\nnorm{\v{\beta}}_q \leq \rho$. This can be reformulated as
		\begin{equation}
			\label{eq:grr_manually_defining_containment}
			\min_{\mat{H}\v{\beta} = \v{p} - \v{d}} \nnorm{\v{\beta}}_q \leq \rho.
		\end{equation}
		Therefore, $\Ein \subseteq \rho \cdot(\Ecirc - \v{d}) + \v{d}$ can be verified by taking the maximum of the left-hand side of \eqref{eq:grr_manually_defining_containment} for $\v{p} = \mat{G}\v{\alpha} + \v{c}$ over all $\v{\alpha}\in\reals^{m}$ satisfying $\norm{\v{\alpha}}_p \leq 1$.
		
		If $\Ecirc$ is degenerate, we have to distinguish two cases:
		\begin{itemize}
			\item If \eqref{eq:rank_constraint} holds, then $\text{Im}\big(\big[\begin{matrix}\mat{G} & \v{c}-\v{d}\end{matrix}\;\big]\big) \not\subseteq \text{Im}(\mat{H})$. This means there exists an $\v{\alpha}$ with $\norm{\v{\alpha}}_{p}\leq 1$ such that there is no $\v{\beta}$ solving
			\begin{equation}
				\label{eq:equation_constraint_ellipsoids}
				\mat{H}\v{\beta} = \mat{G}\v{\alpha} + \v{c} - \v{d}.
			\end{equation}
			Thus, the minimum in \eqref{eq:ellipsotope_containment} is taken over an empty set, which by convention yields the value $\infty$, so $r(\Ein,\Ecirc) = \infty$. This corresponds to the case where $\Ein$ does not lie in the same affine subspace as $\Ecirc$ (see also \cite[Theorem 16.4]{roman_advanced_2008}), and thus there does not exist any scalar $r>0$ such that $\Ein \subseteq r \cdot(\Ecirc - \v{d}) + \v{d}$, so that we may write $r(\Ein, \Ecirc) = \infty$.
			\item If \eqref{eq:rank_constraint} does not hold, $\text{Im}\big(\big[\begin{matrix}\mat{G} & \v{c}-\v{d}\end{matrix}\;\big]\big) \subseteq \text{Im}(\mat{H})$ follows, which means both ellipsotopes lie in the same affine subspace. This can be reduced to the case where $\Ecirc$ is non-degenerate, by projecting both ellipsotopes onto the range of $\mat{H}$. Specifically, if $k = \rank(\mat{H})$ and $\mat{H} = \mat{U}\,\mat{\Sigma}\,\mat{V}^\top$ is the singular value decomposition, we have
			\begin{equation}
				r(\Ein, \Ecirc) = r(\mat{P}\Ein, \mat{P}\Ecirc), \text{ with } \mat{P} = \begin{bmatrix}\mat{I}_k & \mat{0}_{k\times n}\end{bmatrix} \mat{U}^\top.
			\end{equation}
			This is because, if $\mat{H}$ does not have full rank, $\mat{\Sigma}$ has the form
			\begin{equation}
				\mat{\Sigma} = \begin{bmatrix}
					\mat{D} & \mat{0}\\
					\mat{0} & \mat{0}
				\end{bmatrix},
			\end{equation}
			for some diagonal matrix $\mat{D} \in \reals^{k\times k}$.
		\end{itemize}
	\end{proof}
	As shown in \cite[Corollary 3.2.]{kulmburg2023generalized}, if $\Ecirc$ is non-degenerate, the expression \eqref{eq:ellipsotope_containment} can also be computed through\footnote{For clarity, in \cite{kulmburg2023generalized} the notation is chosen so that one considers generalized $p \mapsto q$-norms. Therefore, for the remainder of the present article, whenever the reader is directed to a result from \cite{kulmburg2023generalized}, values of $p$ and $q$ in the present document should be transformed to $p \gets q^*$ and $q \gets p^*$ for \cite{kulmburg2023generalized}.}
	\begin{equation}
		\label{eq:ellipsotope_containment_equiv}
		r(\Ein,\Ecirc) = \max_{\norm{\mat{H}^\top\v{x}}_{q^*}\leq 1} \norm{\mat{G}^\top\v{x}}_{p^*} + \v{x}^\top(\v{c}-\v{d}),
	\end{equation}
	which is similar to a generalized matrix norm as described in \cite{kulmburg2023generalized}. Available strategies to compute such matrix norms \eqref{eq:ellipsotope_containment_equiv} differ based on the values of $p$ and $q$, making more precise approximations possible, under specific conditions. We distinguish between the following cases, which we address separately in the following sub-sections:
	\begin{itemize}
		\item The \textbf{symmetric-V-polytope-in-ellipsotope} problem when $p = 1$, i.e., $\Ein$ is a centrally symmetric V-polytope, is presented in Section \ref{sec:polytope-ellipsotope}.
		\item The \textbf{ellipsoid-in-ellipsoid} problem when $p=q=2$, i.e., $\Ein$ and $\Ecirc$ are ellipsoids, is presented in Section \ref{sec:ellipsoid-ellipsoid}.
		\item The \textbf{zonotope-in-ellipsotope} problem when $p=\infty$, i.e., $\Ein$ is a zonotope, is presented in Section \ref{sec:zonotope-ellipsotope}.
		\item The \textbf{general ellipsotope-in-ellipsotope} problem, which encompasses all other cases for $p$ and $q$, is presented in Section \ref{sec:general_ellipsotope-ellipsotope}.
	\end{itemize}
	\subsection{The Symmetric-V-Polytope-in-Ellipsotope Containment Problem}
	\label{sec:polytope-ellipsotope}
	The function
	\begin{equation*}
		\nnorm{\mat{G}\v{\alpha} + \v{c}-\v{d}\,}_{\Ecirc} = \min_{\mat{H}\v{\beta} = \mat{G}\v{\alpha} + \v{c} - \v{d}}\nnorm{\v{\beta}}_q
	\end{equation*}
	is convex in $\v{\alpha}$, so by the Bauer Maximum principle it suffices to evaluate $\nnorm{\mat{G}\v{\alpha} + \v{c}-\v{d}\,}_{\Ecirc}$ for all $\v{\alpha} = \pm\v{e}_i$, where $i=1,...,m$. Those are $2m$ points, each of which requires solving a nonlinear but convex optimization problem. According to Remark \ref{rmk:optimization_complexity}, this entire procedure can be done using $\mathcal{O}(ml^{3.5}\ln(ml))$ operations, thus this problem is in $\FPTAS$.
	\subsection{The Ellipsoid-in-Ellipsoid Containment Problem}
	\label{sec:ellipsoid-ellipsoid}
	The containment problem for ellipsoids has been studied in \cite{boyd_convex_2004, ellipsoids2022}. Similar ideas can be used to compute $r(\Ein, \Ecirc)$:
	\begin{theorem}
		\label{thm:ellipsoid-ellipsoid}
		If $\Ecirc$ is non-degenerate,
		\begin{equation}
			\label{eq:ellipsoid-ellipsoid}
			r(\Ein, \Ecirc) = \sqrt{\min_{\rho \in S} \rho},
		\end{equation}
		where $S$ is the (convex) set of elements $\rho$ satisfying
		\begin{equation}
			\label{eq:ellipsoid_rho_condition}
			\begin{pmatrix}\mat{\Theta}^\top\mat{\Theta} & \mat{\Theta}^\top\v{\theta}\\ \v{\theta}^{\,\top}\mat{\Theta} & \v{\theta}^{\,\top}\v{\theta} \end{pmatrix} - \rho\begin{pmatrix}\mat{0}_{m\times m} & \v{0}_m\\\v{0}_m^\top & 1\end{pmatrix} \preceq \delta \begin{pmatrix}\mat{I}_m & \v{0}_m\\\v{0}_m^\top & -1\end{pmatrix}
		\end{equation}
		for some $\delta \geq 0$, where $\mat{\Theta} = \mat{H}^+\mat{G}$ and $\v{\theta} = \mat{H}^+(\v{c} - \v{d})$. The optimization problem in \eqref{eq:ellipsoid-ellipsoid} is a semidefinite program that can be solved using $\mathcal{O}((m+1)^{4.5}\ln(m+1))$ operations according to Remark \ref{rmk:optimization_complexity}, thus this problem is in $\FPTAS$.
	\end{theorem}
	
	\begin{proof}
		We consider solutions for $\v{\beta}$ to the constraint in \eqref{eq:equation_constraint_ellipsoids}: By \cite[p. 112, Theorem 2]{james_generalised_1978}, $\v{\beta}$ can be expressed as
		\begin{equation*}
			\v{\beta} = \mat{H}^+(\mat{G}\v{\alpha} + \v{c} - \v{d}) + (\mat{I}_{l} - \mat{H}^+\mat{H})\v{\omega},
		\end{equation*}
		where $\v{\omega}\in\reals^{l}$ is arbitrary. Plugging this into \eqref{eq:ellipsotope_containment} yields
		\begin{equation*}
			r(\Ein,\Ecirc) = \max_{\norm{\v{\alpha}}_2\leq 1} \min_{\v{\omega}}\nnorm{\mat{H}^+(\mat{G}\v{\alpha} + \v{c} - \v{d}) + (\mat{I}_{l} - \mat{H}^+\mat{H})\v{\omega}}_2.
		\end{equation*}
		By \cite[p. 183]{planitz_3_1979}, using $\mat{F} := \mat{I}_{l} - \mat{H}^+\mat{H}$,
		\begin{equation*}
			r(\Ein,\Ecirc) = \max_{\norm{\v{\alpha}}_2\leq 1} \nnorm{\mat{H}^+(\mat{G}\v{\alpha} + \v{c} - \v{d}) + \mat{F}\,\mat{F}^+\mat{H}^+(\mat{G}\v{\alpha} + \v{c} - \v{d})}_2.
		\end{equation*}
		This can be simplified using the identities
		\begin{equation*}
			\mat{F}^+\mat{H}^+ = (\mat{H}\,\mat{F})^+ = (\mat{H} - \mat{H}\,\mat{H}^+\mat{H})^+ = \mat{0},
		\end{equation*}
		which yields
		\begin{equation*}
			\begin{split}
				&r(\Ein,\Ecirc) = \max_{\norm{\v{\alpha}}_2\leq 1} \nnorm{\mat{H}^+(\mat{G}\v{\alpha} + \v{c} - \v{d})}_2\\
				\Leftrightarrow \quad & r(\Ein,\Ecirc)^2 = \max_{\norm{\v{\alpha}}_2\leq 1} \v{\alpha}^\top\mat{\Theta}^\top\mat{\Theta}\v{\alpha} + 2\v{\theta}^\top\mat{\Theta}\v{\alpha} + \v{\theta}{\,\top}\v{\theta}.
			\end{split}
		\end{equation*}
		In other words, we are searching for the smallest scalar $\rho = r^2$ satisfying
		\begin{equation}
			\label{eq:ellipsoids_pre_alternative}
			\forall \v{\alpha}\in\reals^{m}, \quad \v{\alpha}^\top\v{\alpha}\leq 1 \Longrightarrow \v{\alpha}^\top\mat{\Theta}^\top\mat{\Theta}\v{\alpha} + 2\v{\theta}^{\,\top}\mat{\Theta}\v{\alpha} + \v{\theta}^{\,\top}\v{\theta} \leq \rho.
		\end{equation}
		According to the theorem of alternatives (see \cite[Appendix B.2, p. 655]{boyd_convex_2004}), \eqref{eq:ellipsoids_pre_alternative} is equivalent to \eqref{eq:ellipsoid_rho_condition}, which completes the proof.
	\end{proof}
	
	\subsection{The Zonotope-in-Ellipsotope Containment Problem}
	\label{sec:zonotope-ellipsotope}
	If $\Ein$ is a zonotope $\Zin$, we can simplify \eqref{eq:ellipsotope_containment} by reducing it to the case where $\Zin$ and $\Ecirc$ have center $\v{0}$:
	\begin{lemma}
		Let $\Zin = Z(\mat{G}, \v{c})$ and $\Ecirc = E_q(\mat{H}, \v{d})$ for $\v{c},\v{d}\in\reals^n$, $q\in [1,\infty]$, and $\mat{G}\in\reals^{n\times m}$, $\mat{H}\in\reals^{n\times l}$. Then
		\begin{equation}
			r(\Zin, \Ecirc) = r(Z(\mat{G}'), E_q(\mat{H})),
		\end{equation}
		where $\mat{G}' = \begin{bmatrix}\mat{G} & \v{c}-\v{d}\end{bmatrix}$.
	\end{lemma}
	
	\begin{proof}
		We rewrite \eqref{eq:ellipsotope_containment} using the ellipsotope norm introduced in Proposition \ref{prop:ellipsotope_norm}:
		\begin{equation}
			r(\Zin,\Ecirc) = \max_{\norm{\v{\alpha}}_{\infty}\leq 1} \norm{\mat{G}\v{\alpha} + \v{c} - \v{d}\,}_{E_q(\mat{H})}.
		\end{equation}
		Both $\norm_{E_q(\mat{H})}$ and $\norm_{\infty}$ are symmetric functions, so applying the variable transformation $\v{\alpha} \mapsto -\v{\alpha}$ shows
		\begin{align*}
			&\max_{\norm{\v{\alpha}}_{\infty}\leq 1}\norm{\mat{G}\v{\alpha} + (\v{c}-\v{d})}_{E_q(\mat{H})} \\
			= &\max_{\norm{\v{\alpha}}_{\infty}\leq 1}\norm{\mat{G}\v{\alpha} - (\v{c}-\v{d})}_{E_q(\mat{H})},
		\end{align*}
		from which we deduce
		\begin{equation*}
			r(\Zin,\Ecirc) = \max_{\substack{\norm{\v{\alpha}}_{\infty}\leq 1\\\gamma\in\{\pm1\}}}\norm{\mat{G}\v{\alpha} + \gamma(\v{c}-\v{d})}_{E_q(\mat{H})}.
		\end{equation*}
		The function $\nnorm{\mat{G}\v{\alpha} + \gamma(\v{c}-\v{d})}_{E_q(\mat{H})}$ is convex in $\gamma$, thus by the Bauer maximum principle its maximum over $\gamma \in \{\pm 1\}$ is the same as the maximum over $\gamma \in [-1,1]$, so
		\begin{equation*}
			\begin{split}
				r(\Zin,\Ecirc) &= \max_{\substack{\v{x} = \mat{G}'\v{\alpha}'\\\nnorm{\v{\alpha}'}_{\infty}\leq 1}} \nnorm{\v{x}}_{E_q(\mat{H})} = r(Z(\mat{G}'), E_q(\mat{H})).
			\end{split}
		\end{equation*}
	\end{proof}
	
	\subsubsection{Linear Relaxation}
	\label{subsec:presentation_st}
	If $\Ecirc$ is also a zonotope $\Zcirc$, we have the following approximation of $r(\Zin,\Zcirc)$ \cite[Corollary 4]{sadraddini_linear_2019}:
	\begin{equation}
		\ST_{\infty}(\mat{G}, \mat{H}) := \min_{\mat{H}\,\mat{X}=\mat{G}}\norm{\mat{X}}_{\infty}.
	\end{equation}
	The authors showed that if $\ST_{\infty}(\mat{G}, \mat{H}) \leq 1$, then $Z(\mat{G}) \subseteq Z(\mat{H})$, however no conclusion could be made if $\ST_{\infty}(\mat{G}, \mat{H}) > 1$. This is equivalent to
	\begin{equation}
		\label{eq:ZC<ST}
		r(\Zin,\Zcirc) \leq \ST_{\infty}(\mat{G}, \mat{H}).
	\end{equation}
	We can generalize this relaxation to cases where $\Zcirc$ is an ellipsotope using results from \cite{kulmburg2023generalized}, and additionally we can also provide a bound on the accuracy of the algorithm:
	\begin{theorem}
		\label{thm:main_thm}
		Let $\mat{G} \in \reals^{n\times m}$, $\mat{H} \in \reals^{n\times l}$, and assume $\mat{H}$ is surjective. Then
		\begin{equation}
			\label{eq:ST_def}
			\ST_q(\mat{G}, \mat{H}) := \min_{\mat{H}\,\mat{X}=\mat{G}}\norm{\mat{X}}_{L_{1,q}^\top}
		\end{equation}
		satisfies
		\begin{equation}
			\label{eq:r_equivalence_proof}
			\begin{split}
				r(Z(\mat{G}),E_q(\mat{H})) &\leq \ST_q(\mat{G}, \mat{H})\\
				&\leq \frac{\gamma_{q^*}\sqrt{m}}{\gamma_1} r(Z(\mat{G}),E_q(\mat{H})),
			\end{split}
		\end{equation}
		where $\gamma_{k}$ is the $k$-th root of the $k$-th moment of a standard normal distribution, i.e.,
		\begin{equation}
			\label{eq:gaussian_moment}
			\gamma_k = \left(\frac{2^{k/2}}{\sqrt{\pi}}\Gamma\left(\frac{k+1}{2}\right)\right)^{1/k},
		\end{equation}
		and $\Gamma$ is the Gamma function.
		Equivalently, we have that
		\begin{itemize}
			\item if $\ST_q(\mat{G}, \mat{H}) \leq 1$, then $Z(\mat{G}) \subseteq E_q(\mat{H})$.
			\item if $\ST_q(\mat{G}, \mat{H}) > \frac{\gamma_{q^*}\sqrt{m}}{\gamma_1}$, then $Z(\mat{G}) \not\subseteq E_q(\mat{H})$.
		\end{itemize}
	\end{theorem}
	
	\begin{proof}
		If $\mat{H}$ is surjective, combining \eqref{eq:ellipsotope_containment_equiv} and \cite[Theorem 5.9.]{kulmburg2023generalized}, yields \eqref{eq:r_equivalence_proof}.
		Since $Z(\mat{G}) \subseteq E_q(\mat{H})$ holds if and only if $r(Z(\mat{G}),E_q(\mat{H})) \leq 1$, the first inequality in \eqref{eq:r_equivalence_proof} shows that $\ST_q(\mat{G}, \mat{H}) \leq 1$ implies $r(Z(\mat{G}),E_q(\mat{H}))\leq 1$, and thus $Z(\mat{G}) \subseteq E_q(\mat{H})$. Conversely, if $\ST_q(\mat{G}, \mat{H}) > \frac{\gamma_{q^*}\sqrt{m}}{\gamma_1}$, the second inequality in \eqref{eq:r_equivalence_proof} implies $Z(\mat{G}) \not\subseteq E_q(\mat{H})$.
	\end{proof}
	
	Note that $\ST_q(\mat{G}, \mat{H})$ can be computed in polynomial time using $\mathcal{O}((ml)^{3.5}\ln(ml))$ operations according to Remark \ref{rmk:optimization_complexity}. In particular, for $q=1$ and $q=\infty$, linear programming can be used for \eqref{eq:ST_def}, whereas convex quadratic optimization can solve \eqref{eq:ST_def} for $q=2$.
	
	\medskip
	
	\subsubsection{Exactness of the Linear Relaxation}
	Beyond the bound given in Theorem \ref{thm:main_thm}, in many cases $\ST_q(\mat{G}, \mat{H}) = r(\Zin,\Ecirc)$. This is particularly apparent for the zonotope-in-zonotope containment problem \cite{sadraddini_linear_2019}.
	Before we can introduce the next theorem, we need to deduce the dual formulation of $\ST_q$: By \cite[Proposition 3.1]{kulmburg2023generalized} and \cite[Lemma 2.3]{kulmburg2023generalized}
	\begin{equation}
		\label{eq:ST_dual}
		\ST_q(\mat{G}, \mat{H}) = \max_{\|\mat{H}^\top\mat{Y}\|_{L_{\infty,q^*}^\top}\leq1}\trace(\mat{G}^\top\mat{Y}).
	\end{equation}
	Note that \eqref{eq:ST_dual} can still be solved in polynomial time, but this time using $\mathcal{O}((mn)^{3.5}\ln(mn))$ operations, since the variable matrix $\mat{X}$ is an $n\times m$-matrix.
	\begin{theorem}
		\label{thm:resonances}
		Let $\mat{G} \in \reals^{n\times m}$, $\mat{H} \in \reals^{n\times l}$, and suppose
		\begin{equation}
			\label{eq:st_dual_repeat}
			\mat{Y} = \argmax_{\|\mat{H}^\top\mat{Y}'\|_{L_{\infty,q^*}^\top}\leq1}\trace(\mat{G}^\top\mat{Y}').
		\end{equation}
		If all columns $\v{y}_i$ of $\mat{Y}$ are identical up to their sign, then $\ST_q(\mat{G},\mat{H}) = r(Z(\mat{G}), E_q(\mat{H}))$. In other words:
		\begin{equation}
			\label{eq:condition_resonance}
			\begin{split}
				&\forall i,j \quad \v{y}_i = \v{y}_j \text{ or } \v{y}_i = -\v{y}_j\\
				& \quad \Rightarrow \ST_q(\mat{G}, \mat{H}) = r(Z(\mat{G}), E_q(\mat{H})).
			\end{split}
		\end{equation}
		In particular, if $q=\infty$ and $\mat{H}$ is square and invertible (i.e., $\Ecirc$ is a parallelotope), condition \eqref{eq:condition_resonance} is always satisfied.
	\end{theorem}
	
	\begin{proof}
		Suppose $\mat{Y}$ maximizes \eqref{eq:st_dual_repeat}, so that
		\begin{equation}
			\label{eq:condition_X_resonancing}
			\nnorm{\mat{H}^\top\mat{Y}}_{L_{\infty,q^*}^\top}\leq 1.
		\end{equation}
		If $\mat{Y}$ satisfies \eqref{eq:condition_resonance}, $\mat{Y} = \v{y}\v{\sigma}^\top$ for some $\v{y}\in\reals^n$ and $\v{\sigma}\in \{\pm1\}^l$, so using \eqref{eq:condition_X_resonancing} we get
		\begin{equation}
			\label{eq:detailled_constraint_resonancing}
			\begin{split}
				1 \geq \left[\sum_i \max_j |\v{h}_i^\top\v{y}\sigma_j|^{q^*}\right]^{1/q^*} = \nnorm{\mat{H}^\top\v{y}}_{q^*},
			\end{split}
		\end{equation}
		which means that
		\begin{align*}
			\ST_q(\mat{G}, \mat{H}) &= \trace(\mat{G}^\top\mat{Y}) && \text{(definition of $\mat{Y}$)}\\
			&= \trace(\mat{G}^\top\v{y}\v{\sigma}^\top) && \text{($\mat{Y} = \v{y}\v{\sigma}^\top$)}\\
			&= \trace(\v{\sigma}^\top\mat{G}^\top\v{y}) &&\text{($\trace(\mat{M}\,\mat{L}) = \trace(\mat{L}\,\mat{M})$)}\\
			&\leq \max_{\substack{\norm{\v{s}}_{\infty}\leq1\\\norm{\mat{H}^\top\v{z}}_{q^*}\leq1}}\v{s}^\top\mat{G}^\top\v{z} & \begin{split}&\text{(definition of $\max$}\\ & \text{using \eqref{eq:detailled_constraint_resonancing})}\end{split}\\
			&= \max_{\norm{\mat{H}^\top\v{z}}_{q^*}\leq1}\nnorm{\mat{G}^\top\v{z}}_1 &&\text{(Dual of the $\infty$-norm)}\\
			&= r(Z(\mat{G}), E_q(\mat{H})) && \text{(using \eqref{eq:ellipsotope_containment_equiv})}
		\end{align*}
		where for the last equality we used duality and \eqref{eq:ellipsotope_containment_equiv}. This proves  $r(\Zin, \Ecirc) \geq \ST_q(\mat{G}, \mat{H})$; since $r(\Zin, \Ecirc) \leq \ST_q(\mat{G}, \mat{H})$ by Theorem \ref{thm:main_thm}, $\ST_q(\mat{G}, \mat{H}) = r(\Zin, \Ecirc)$ follows. As for the special case when $\Ecirc$ is a parallelotope, using the variable transformation $\mat{Z}' = \mat{H}^\top\mat{Y}'$ reveals $\mat{Y} = \left(\mat{H}^\top\right)^{-1}\mat{Z}$, with
		\begin{equation}
			\label{eq:parallelotope_st_Z}
			\mat{Z} = \argmax_{\|\mat{Z}'\|_{L_{\infty,1}^\top}\leq1}\trace\left(\mat{G}^\top\left(\mat{H}^{\top}\right)^{-1}\mat{Z}'\right).
		\end{equation}
		By the Bauer maximum principle, the maximum in \eqref{eq:parallelotope_st_Z} is achieved at one of the extreme points of the constraints $\|\mat{Z}'\|_{L_{\infty,1}^\top} \leq 1$. Those extreme points can easily be seen to be of the form $\v{e}_i\v{\sigma}^\top$ for some $i=1,...,n$ and $\v{\sigma}\in \{\pm1\}^m$, which proves $\mat{Y}$ must satisfy \eqref{eq:condition_X_resonancing}.
	\end{proof}
	
	\subsubsection{Semidefinite Relaxation}
	For $q\in(1,2]$, while the approximation from Theorem \ref{thm:main_thm} remains valid, the following one is tighter:
	\begin{theorem}
		\label{thm:zono_in_ellipsoid}
		Let $\mat{G}\in\reals^{n\times m}$, $\mat{H}\in\reals^{n\times l}$, where $\mat{H}$ is surjective. For $q\in(1,2]$, we define
		\begin{equation}
			\label{eq:ZSR_def}
			\ZSR_q(\mat{G}, \mat{H}) := \frac{1}{2}\min_{(\v{v},\v{w})\in K} \norm{\v{v}}_1+\norm{\v{w}}_{\frac{q}{2-q}},
		\end{equation}
		where $K$ is the set of tuples $(\v{v}, \v{w}) \in \reals^{l}\times \reals^{m}$ satisfying
		\begin{equation}
			\begin{bmatrix}\Diag(\v{v}) & \mat{0}\\\mat{0} & \mat{H}\Diag(\v{w})\mat{H}^\top\end{bmatrix} \preceq \begin{bmatrix}\mat{0} & \mat{G}^\top\\\mat{G} & \mat{0}\end{bmatrix}.
		\end{equation}
		Then, $\ZSR_q$ satisfies
		\begin{equation}
			r(Z(\mat{G}),E_q(\mat{H})) \leq \ZSR_q(\mat{G}, \mat{H}) \leq \frac{\gamma_{q^*}}{\gamma_{1}} r(Z(\mat{G}),E_q(\mat{H})),
		\end{equation}
		with $\gamma_s$ being the $s$-th root of the $s$-th Gaussian moment as in \eqref{eq:gaussian_moment}.
		Equivalently, we have that
		\begin{itemize}
			\item if $\ZSR_q(\mat{G}, \mat{H}) \leq 1$, then $Z(\mat{G}) \subseteq E_q(\mat{H})$;
			\item if $\ZSR_q(\mat{G}, \mat{H}) > \gamma_{q^*}/\gamma_{1}$, then $Z(\mat{G}) \not\subseteq E_q(\mat{H})$.
		\end{itemize}
	\end{theorem}
	\begin{proof}
		Similarly to Theorem \ref{thm:main_thm}, this follows from \cite[Theorem 5.6.]{kulmburg2023generalized}.
	\end{proof}
	One can compute $\ZSR_q$ with a solver for semidefinite problems, using $\mathcal{O}(n^{4.5}\ln(n))$ operations according to Remark \ref{rmk:optimization_complexity}.
	\begin{figure}[t!]
		\begin{flushleft}
			\definecolor{mycolor1}{rgb}{0.89020,0.10588,0.13725}%
			\definecolor{mycolor2}{rgb}{0.00000,0.36078,0.67059}%
			\definecolor{mycolor3}{rgb}{1.00000,0.76471,0.14510}%
%
		\end{flushleft}
		\caption{An example of the zonotope-in-ellipsoid containment problem. The exact value of $r(\Zin, \Ecirc)$ is difficult to compute, but $\widehat{r}$ as defined in Example \ref{ex:zono_in_ellipsoid} yields error bounds of the actual solution.}
		\label{fig:zono_in_ell}
	\end{figure}
	\begin{example}
		\label{ex:zono_in_ellipsoid}
		Let $\Ecirc = E(\mat{I}_2)$ be the unit disk in $\reals^2$, and $\Zin = Z(\mat{G})$ with
		\begin{equation*}
			\mat{G} = \frac{1}{100}%
,
		\end{equation*}
		and let $\widehat{r} = \ZSR_2(\mat{G}, \mat{I}_2)$. A numerical evaluation yields $\widehat{r} \approx 3.89$, and using the algorithm from \cite[Section III]{kulmburgSearchbasedStochasticSolutions2024} we get $r(\Zin, \Ecirc) \approx 3.67$, satisfying $\sqrt{2/\pi} \widehat{r} \leq r(\Zin, \Ecirc) \leq \widehat{r}$ as per Theorem \ref{thm:zono_in_ellipsoid}. By definition, $r(\Zin, \Ecirc) \Ecirc$ tightly contains $\Zin$. However, as we will see in Section \ref{sec:hardness}, $r(\Zin, \Ecirc)$ is $\NP$-hard to compute exactly, so $\widehat{r}$ can serve as an approximation. Specifically, Theorem \ref{thm:zono_in_ellipsoid} implies that any set contained in $\sqrt{2/\pi}\widehat{r} \Ecirc$ must also be contained in $\Zin$, whereas any set that has a point outside $\widehat{r} \Ecirc$ can not possibly be contained in $\Zin$. 
	\end{example}
	\subsection{The General Ellipsotope-in-Ellipsotope Containment Problem}
	\label{sec:general_ellipsotope-ellipsotope}
	For other values of $p, q$ that we have not treated above, if the inbody $\Ein$ and the circumbody $\Ecirc$ have the same center, there are two cases:
	\subsubsection{$1 < q \leq 2 \leq p < \infty$}
	\label{sec:1_q_2_p_infty}
	In this case, we can use \cite[Section 5.1.]{kulmburg2023generalized} to find a constant-ratio approximation to $r(\Ein, \Ecirc)$:
	\begin{theorem}
		\label{thm:1_q_2_p_infty}
		Let $\Ein = E_p(\mat{G}) \subseteq \reals^n$ and $\Ecirc = E_q(\mat{H})\subseteq \reals^{n}$ for $1 < q \leq 2 \leq p < \infty$, where $\mat{H}$ is surjective. We define
		\begin{equation}
			\label{eq:SR_def}
			\SR_{p,q}(\mat{G}, \mat{H}) := \frac{1}{2}\min_{(\v{v},\v{w})\in K} \norm{\v{v}}_{\frac{p}{p-2}}+\norm{\v{w}}_{\frac{q^*}{q^*-2}},
		\end{equation}
		where $K$ is the set of tuples $(\v{v}, \v{w}) \in \reals^{l}\times \reals^{m}$ satisfying
		\begin{equation}
			\begin{bmatrix}\Diag(\v{v}) & \mat{0}\\\mat{0} & \mat{H}\Diag(\v{w})\mat{H}^\top\end{bmatrix} \preceq \begin{bmatrix}\mat{0} & \mat{G}^\top\\\mat{G} & \mat{0}\end{bmatrix}.
		\end{equation}
		Then, $\SR_{p,q}$ satisfies
		\begin{equation}
			r(\Ein,\Ecirc) \leq \SR_{p,q}(\mat{G}, \mat{H}) \leq \gamma_{q^*}\gamma_{p} r(\Ein,\Ecirc),
		\end{equation}
		with $\gamma_k$ being the $k$-th root of the $k$-th Gaussian moment as in \eqref{eq:gaussian_moment}.
		Equivalently, we have that
		\begin{itemize}
			\item if $\SR_{p,q}(\mat{G}, \mat{H}) \leq 1$, then $\Ein \subseteq \Ecirc$;
			\item if $\SR_{p,q}(\mat{G}, \mat{H}) > \gamma_{q^*}\gamma_{p}$, then $\Ein \not\subseteq \Ecirc$.
		\end{itemize}
	\end{theorem}
	\begin{proof}
		Once again, this can be proven in the same way as Theorem \ref{thm:main_thm} and Theorem \ref{thm:zono_in_ellipsoid}, this time using \cite[Theorem 5.4.]{kulmburg2023generalized}.
	\end{proof}
	One can compute $\SR_{p,q}$ with a solver for semidefinite problems, using $\mathcal{O}(n^{4.5}\ln(n))$ operations according to Remark \ref{rmk:optimization_complexity}.
	
	\medskip
	
	\subsubsection{Remaining Cases}
	As we will discuss in the next section, there are remaining cases because, to the best of our knowledge, the apparently simpler problem of computing the matrix norm
	\begin{equation*}
		\max_{\norm{\v{x}}_{q*} \leq 1} \norm{\mat{A}\v{x}}_{p^*}
	\end{equation*}
	of a matrix $\mat{A}$ has not yet been solved in the literature, which corresponds to the special case $\Ein = E_p(\mat{I})$, $\Ecirc = E_q(\mat{A}^\top)$. Nevertheless, it should be mentioned that if $\Ein$ and $\Ecirc$ have the same center, one can construct rough approximations using the equivalence of norms: by \cite[Proposition 37.6]{tomczak-jaegermann_banach-mazur_1989}, for any $s,t \in [1,\infty]$ such that $s < t$, and for any $\v{x}\in\reals^n$ we have that
	\begin{equation}
		\norm{\v{x}}_t \leq \norm{\v{x}}_s \leq n^{1/s - 1/t} \norm{\v{x}}_t,
	\end{equation}
	where we use the convention $1/\infty = 0$.
	Consequently, we can use the equivalence of norms to reduce the containment problem to one of the cases we treated above.
	
	\section{Hardness of the Ellipsotope Containment Problem}
	\label{sec:hardness}
	\begin{figure}[t!]
		\centering
		\begin{tikzpicture}[scale=0.5]
			
			\draw[-] (0,0) -- (10,0) -- (10,10) -- (0,10) -- (0,0);
			
			\draw[-] (0,10) -- (-0.1-\padding,10);
			\draw[-] (0,5) -- (-0.1-\padding,5);
			\draw[-] (0,0) -- (-0.1-\padding,0);
			
			\draw[-] (0,0) -- (0,-0.1-\padding);
			\draw[-] (5,0) -- (5,-0.1-\padding);
			\draw[-] (10,0) -- (10,-0.1-\padding);
			
			\draw[->] (0,10) -- (0,11);
			\draw[->] (10,0) -- (11,0);
			
			\fill[pattern=my large dots, pattern color=Red] (0,0) rectangle (5,5);
			\fill[color=Red, opacity=0.3] (0,0) rectangle (5,5);
			\fill[color=Red] (-\padding/2,-\padding/2) rectangle (5-\padding/2,\padding/2);
			
			\fill[pattern=my large dots, pattern color=Red] (5,5) rectangle (10,10);
			\fill[color=Red, opacity=0.3] (5,5) rectangle (10,10);
			\fill[color=Red] (10-\padding/2,5+\padding/2) rectangle (10+\padding/2,10+\padding/2);
			
			\fill[pattern=crosshatch, pattern color=Blue] (5,0) rectangle (10,5);
			\fill[color=Blue, opacity=0.3] (5,0) rectangle (10,5);
			\fill[color=Blue] (5-\padding/2,-\padding/2) rectangle (5+\padding/2,5+\padding/2);
			\fill[color=Blue] (5-\padding/2,5-\padding/2) rectangle (10+\padding/2,5+\padding/2);
			\fill[color=Blue] (5-\padding/2,-\padding/2) rectangle (10+\padding/2,\padding/2);
			\fill[color=Blue] (10-\padding/2,-\padding/2) rectangle (10+\padding/2,5+\padding/2);
			
			\fill[color=Yellow] (-\padding/2,10+\padding/2) rectangle (\padding/2,-\padding/2);
			\fill[color=Yellow] (5,5) circle (\padding);
			
			\fill[color=black] (\padding/2,10-\padding/2) rectangle (5+\padding/2,10+\padding/2);
			\fill[color=Red] (5+\padding/2,10-\padding/2) rectangle (10+\padding/2,10+\padding/2);

			\draw (0,11.5) node {$\bm{q}$};
			\draw (11.5,0) node {$\bm{p}$};
			
			\draw (-0.5-\padding,10) node {$\infty$};
			\draw (-0.5-\padding,5) node {$2$};
			\draw (-0.5-\padding,0) node {$1$};
			
			\draw (10,-0.5-\padding) node {$\infty$};
			\draw (5,-0.5-\padding) node {$2$};
			\draw (0,-0.5-\padding) node {$1$};
			
			\draw (-1.5-\padding,10) node[rotate=90] {Zonotope};
			\draw (-1.5-\padding,5) node[rotate=90] {Ellipsoid};
			\draw (-1.5-\padding,0) node[rotate=90] {V-Polytope};
			
			\draw (10,-1.5-\padding) node {Zonotope};
			\draw (5,-1.5-\padding) node {Ellipsoid};
			\draw (0,-1.5-\padding) node {V-Polytope};

			\draw (5,-2.5-\padding) node {\textbf{Inbody $\Ein$}};
			
			\draw (-2.5-\padding,5) node[rotate=90] {\textbf{Circumbody $\Ecirc$}};

			\fill[color=Yellow] (-3,-3.5) rectangle (-2,-4);
			\fill[color=black] (-3,-4.5) rectangle (-2,-5);

			\draw (-2,-3.7) node[anchor=west] { : In $\FPTAS$};
			\draw (-2,-4.7) node[anchor=west] { : Not in $\FPTAS$};

			\fill[pattern=crosshatch, pattern color=Blue] (5,-3.5) rectangle (6,-4);
			\fill[color=Blue, opacity=0.3] (5,-3.5) rectangle (6,-4);
			\fill[pattern=my large dots, pattern color=Red] (5,-4.5) rectangle (6,-5);
			\fill[color=Red, opacity=0.3] (5,-4.5) rectangle (6,-5);
			\draw (6,-3.7) node[anchor=west] { : $\tau$-inapproximable};
			\draw (6,-4.7) node[anchor=west] { : Inapproximable};

		\end{tikzpicture}
		\caption{An overview of the complexity of different containment problems. We always assume the most pessimistic case, e.g., for $p \geq q$ and $1 < q < 2$, we displayed the problem is inapproximable, though this would require $\NP \not\subseteq \problem{BPP}$. A solid boundary line means that the boundary is included in the region. For the V-polytope case, we only consider the class of centrally symmetric V-polytopes.}
		\label{fig:containment_square}
	\end{figure}
	We now address the computational complexity of the ellipsotope containment problem for different values of $p$ and $q$. Formally speaking, there are two distinct ways to define the containment problem for ellipsotopes $\Ein$ and $\Ecirc$: either the aim is to verify $\Ein \subseteq \Ecirc$, which is a decision problem that should be answered by 'Yes' or 'No,' or the goal is to compute $r(\Ein, \Ecirc)$, which is an optimization problem that should return a real number.
	
	The latter one makes a more detailed complexity analysis possible, so we will focus on the optimization formulation.
	The following list summarizes results that can be deduced directly from the available literature about matrix norms; a graphical overview is presented in Figure \ref{fig:containment_square}. In some cases, we implicitly use the argument that $\norm{\mat{A}}_{s\mapsto t} = \norm{\mat{A}^\top}_{t^*\mapsto s^*}$ (which can be shown using duality), so that computing the $s\mapsto t$-norm is as hard as computing the $t^*\mapsto s^*$-norm, for $s,t\in[1,\infty]$. In what follows, $\tau$ will denote an appropriate constant that depends solely on $p$ and $q$.
	\begin{enumerate}
		\item If $p = 1$ or $p=q=2$, the problem is in $\FPTAS$ as discussed in Sections \ref{sec:polytope-ellipsotope} and \ref{sec:ellipsoid-ellipsoid}.
		\item If $1\leq q\leq2\leq p \leq \infty$ and $q\neq p$, the problem is $\tau$-inapproximable unless $\Poly = \NP$ by \cite[Theorem 1.4.]{bhattiprolu_2023}.
		\item If $p = \infty$ and $2 < q < \infty$, the problem is inapproximable unless $\Poly = \NP$ by \cite[Theorem 6.4.]{bhattiprolu_2023}.
		\item If $1 < p < 2$ and $q = 1$, the problem is inapproximable unless $\Poly = \NP$ by \cite[Theorem 6.4.]{bhaskara}.
		\item If $p = q$ and $p\in(1,2)$ or $p \in (2,\infty)$, the problem is inapproximable unless $\Poly = \NP$ by \cite[Theorem 6.2.]{bhaskara}\footnote{In fact, \cite[Theorem 6.2.]{bhaskara} claims this also holds for $p=2$, but this is a mistake, as confirmed in the sentence before \cite[Lemma B.1.]{bhaskara}.}.
		\item If $1 < p < 2$ and $q \leq p$, or if $q \leq p$ and $2 < q < \infty$, the problem is inapproximable unless $\Poly=\NP$ by \cite[Theorem 6.3.]{bhaskara}.
		\item If $q \geq p$ and $1 < q < 2$, or if $2 < p < \infty$ and $q \geq p$, the problem is inapproximable unless $\NP \subseteq \problem{BPP}$ by \cite[Theorem 1.1.]{bhattiprolu_2023}.
		\item For other cases with $q < \infty$, only a few specific results are known (for example, \cite[Theorem 2.5]{barak_brandao_et_al} proved the $2\mapsto 4$-norm is inapproximable, assuming that the Exponential Time Hypothesis is true).
	\end{enumerate}
	\begin{remark}
		For certain cases in the list above, more stringent statements may hold. For example, \cite[Theorem 6.2.]{bhaskara} showed the case $p=q \in (1,2)$ is inapproximalbe within a factor of $\Omega(2^{(\log(N))^{1-\delta}})$ for any $\delta>0$ (see Remark \ref{rmk:higher_inapproximability}), under the stronger assumption that $\NP\not\subseteq\problem{DTIME}(2^{\text{polylog}(N)})$, where $N$ is the maximum of the number of generators of $\Ein$ and $\Ecirc$. Since our focus is on approximability, we encourage readers seeking further insights into these stronger claims to consult the referenced literature.
	\end{remark}
	The cases we did not cover are $q=\infty$ and $p\in(1,\infty]$, corresponding to the ellipsotope-in-zonotope case. We showed in \cite{Kulmburg2021} that for $p=q=\infty$ the problem is hard to compute, but we now go one step further and generalize this statement to other values $p \in (1,\infty)$ by using a reduction from the $p\mapsto 1$-norm, which is $\NP$-hard to compute and even $\tau$-inapproximable unless $\Poly = \NP$ as shown in \cite[Theorem 6.4.]{bhaskara} for $1<p<2$, and \cite[Theorem 1.4.]{bhattiprolu_2023} for $2\leq p \leq \infty$.
	Our reduction is loosely inspired by \cite{Kulmburg2021, burger_finding_2000}.
	\begin{theorem}
		\label{thm:not_FPTAS}
		Unless $\Poly = \NP$, the ellipsotope-in-zonotope containment problem for $p$-ellipsotopes with $p\in(1,\infty]$ is not in $\FPTAS$. Furthermore, the corresponding decision formulation is $\coNP$-hard. This remains true, even if the ellipsotope and zonotope are required to have the same center.
	\end{theorem}
	\begin{proof}
		Since the proof is very technical, we present it in Appendix \ref{sec:hardness_ell_in_zono}.
	\end{proof}
	\begin{remark}
		For certain applications (e.g., reachability analysis, as discussed in \cite[p.~12]{althoff_dissertation}), it is useful to consider the complexity of certain operations with respect to the \emph{order} of zonotopes, which is defined as $\varrho = l/n$, for a circumbody zonotope in $\reals^n$ with $l$ generators. The proof of Theorem \ref{thm:not_FPTAS} in the Appendix uses a zonotope of order less than $2$, which means the ellipsotope-in-zonotope containment problem (for $p\in(1,\infty]$) is not in $\FPTAS$ even if we restrict our attention to zonotopes with order $2$.
	\end{remark}
	For the ellipsotope-in-zonotope containment problem when $p\in (2,\infty]$ and $q=\infty$, we can deduce a stronger result under additional assumptions, which require the following complexity classes:
	\begin{definition}[$\RP$ and $\RQP$ {\cite[Chapter 7.3]{Arora_Barak_2009}}]
		The complexity class $\RQP$ (\textbf{R}andomized \textbf{Q}uasi-\textbf{P}olynomial) is the class of problems for which a probabilistic Turing machine exists with the following properties:
		\begin{itemize}
			\item It always runs in quasi-polynomial time (i.e., in time $2^{\mathcal{O}\left(\log(N)^k\right)}$ for some $k\in\naturals$, with respect to the input size $N$).
			\item If the correct answer is 'No,' it always returns 'No.'
			\item If the correct answer is 'Yes,' it returns 'Yes' with a probability of at least $1/2$; otherwise, it returns 'No.'	
		\end{itemize}
		We also define $\RP$ (\textbf{R}andomized \textbf{P}olynomial) similarly by requiring the runtime to be polynomial instead.
	\end{definition}
	The class $\RQP$ is not well-known in the literature, so its relation to $\NP$ is currently unknown. On the other hand, it is known that $\Poly \subseteq \RP \subseteq \NP$, though it is not known whether either of these containments is strict (see also \cite[Chapter 7.3]{Arora_Barak_2009}).
	\begin{theorem}
		\label{thm:RP_RQP_zono_containment}
		Unless $\NP = \RP$, the ellipsotope-in-zonotope containment problem for $p$-ellipsotopes with $p\in(2,\infty]$ is inapproximable. Additionally, assuming the generator matrix of the ellipsotope is an $n\times m$-matrix, in the sense of Remark \ref{rmk:higher_inapproximability}, this problem is inapproximable within a factor $\Omega(2^{(\log \widetilde{N})^{1-\delta}})$ for any $\delta > 0$ unless $\NP \subseteq \RQP$, where $\widetilde{N} = \max\{n,m\}$. These statements remain true, even if the ellipsotope and zonotope are required to have the same center.
	\end{theorem}
	\begin{proof}
		A proof can be found in Appendix \ref{sec:proof_thm_RP_RQP}.
	\end{proof}

	\section{Application to the Computation of Safe Sets}
	\label{sec:applications}
	To demonstrate the applicability of our results, we consider the problem of computing a safe set \cite{gruber} for a controllable linear time-invariant system:
	\begin{equation}
		\label{eq:ode}
		\begin{split}
			&\dot{\v{x}} = \mat{A}\v{x} + \mat{B}\v{u} + \mat{E}\v{w} + \v{\chi},
		\end{split}
	\end{equation}
	where $\mat{A} \in \reals^{n_x \times n_x}$, $\mat{B} \in \reals^{n_x \times n_u}$, $\mat{E} \in \reals^{n_x \times n_w}$, and $\v{\chi} \in \reals^{n_x}$ for $n_x, n_u, n_w\in\naturals$. We call $\v{u}$ the \emph{control input} and $\v{w}$ the \emph{disturbance} over some time interval $t \in [0,\tend]$ with $\tend\geq 0$, and $\v{x}$ is the \emph{state} of the system.
	We impose that $\v{x}$, $\v{u}$, and $\v{w}$ are constrained to the sets $\X \subseteq \reals^{n_x}$, $\U\subseteq \reals^{n_u}$, and $\W\subseteq \reals^{n_w}$, respectively, for all $t\in[0,\tend]$, where $\X$ and $\U$ are polyhedra $P(\mat{\Lambda}_{\X}, \v{\lambda}_{\X})$ and $P(\mat{\Lambda}_{\U}, \v{\lambda}_{\U})$, respectively (though $\U$ is generally assumed to be compact), and $\W$ is a zonotope $Z(\mat{G}_{\W})$ centered at the origin (if it is not centered at the origin, without loss of generality, we can extract the center and add it to $\v{\chi}$).
	
	We consider the setting of sampled-data systems, i.e., we control the continuous-time system by a digital controller \cite{Rakovic2017_sampledData}.
	The state of the system is measured at discrete points in time $t_i = i\tend/\Nts$ for $i=0,...,\Nts$, with $\Nts\in\naturals$ being the number of time steps.
	Based on the initial state $\v{x}_0 \in \reals^{n_x}$ and the state at the current time step, similarly to \cite{gruber}, we choose the state feedback controller
	\begin{equation}
		\label{eq:piecewise_constant_u}
		\begin{split}
			&\v{u}(t,\v{x},\v{x}_0) = \mat{K}\v{x} + \v{c}_{u,i} + \mat{U}_{i}\v{\beta}(\v{x}_0), \text{ for } t \in [t_i, t_{i+1}),
		\end{split}
	\end{equation}
	where $\mat{K}$ is a stabilizing feedback gain, $\v{c}_{u,i} \in \reals^{n_u}$ and $\mat{U}_i\in \reals^{n_u\times m_{u}}$ are optimization variables, and $\v{\beta}(\v{x}_0)$ is a parametrization of $\v{x}_0$ (details can be found in Section \ref{sec:modeling_terminal_region}).
	For our analysis, we compute $\mat{K}$ from a linear quadratic regulator \cite{Lewis2012}.
	
	Our goal is to find a safe set $\T\subseteq \reals^{n_x}$ (see \cite[Section IV.B.]{gruber} and \cite[Theorem 6.]{raghuraman_set_2022}) which is as large as possible, yet satisfies the condition
	\begin{equation}
		\label{eq:terminalRegion}
		\begin{split}
			&\exists \v{c}_{u,i}\in\reals^{n_u}, \mat{U}_{i}\in\reals^{n_u\times m_{u}},\\
			&\;\R_{\v{x}}(\tend; \T) \subseteq \T\\
			&\; \wedge \; \forall t\in[0,\tend], \R_{\v{x}}(t; \T) \subseteq \X\\
			&\; \wedge \; \forall t\in[0,\tend],  \R_{\v{u}}(t; \T) \subseteq \U,\\
		\end{split}
	\end{equation}
	where the sets $\R_{\v{x}}(t; \X_0)$ and $\R_{\v{u}}(t;\X_0)$ are computed as described in \cite[Section III.B.]{gruber} and the containments $\R_{\v{x}}(t; \widehat{\T}) \subseteq \X$ and $\R_{\v{u}}(t; \widehat{\T}) \subseteq \U$ can be solved as in \cite{gruber}.
	In contrast to \cite{gruber}, we will not solve \eqref{eq:terminalRegion} by first computing an invariant target set serving as the circumbody (instead of $\T$) in the second line of \eqref{eq:terminalRegion}. 
	Instead, we will directly solve the containment $\R_{\v{x}}(\tend; \T) \subseteq \T$ using the methods from Section \ref{sec:containment_problem}.
	
	\subsection{Modeling the Terminal Region}
	\label{sec:modeling_terminal_region}
	Similarly to \cite{gruber,schaefer}, $\T$ is modeled either as a zonotope or an ellipsoid.
	If $\T$ is modeled as a zonotope, we leave its center $\v{c}_{\T}$ as an optimization variable and set its generator matrix $\mat{G}_{\T}$ to
	\begin{equation}
		\mat{G}_{\T} = \mat{R}\,\mat{G}_{\mathrm{fixed}}\Diag(\v{s}),
	\end{equation}
	where $s_i \in [0,\infty)$ for $i=1,...,m$ are scaling factors of our optimization variables, $\mat{G}_{\mathrm{fixed}} \in \reals^{n_x\times m}$ is a fixed matrix, and $\mat{R} \in \reals^{n_x \times n_x}$ is a fixed, invertible matrix that can help orient the set. In our case, we compute $\mat{G}_{\mathrm{fixed}}$ by under-approximating the unit hypersphere by a zonotope, using the procedure described in \cite[Lemma 4]{victor_ellipsoid}. As for $\mat{R}$, we discovered empirically that a good choice is
	\begin{equation}
		\label{eq:R_choice}
		\mat{R} = \mat{P}^{-1/2},
	\end{equation}
	where $\mat{P}$ is the unique positive definite solution to the discrete-time algebraic Riccati equation \cite[Sec. 2.4]{Lewis2012}. 
	Intuitively, this makes sense because any sub-level set of the ellipsoid $E(\mat{R})$ is a control invariant set under the linear feedback controller associated with $\mat{P}$.
	
	If $\T$ is modeled as an ellipsoid, we leave its center $\v{c}_{\T}$ as an optimization variable once again and choose the generator matrix
	\begin{equation}
		\mat{G}_{\T} = \mat{R} \Diag(\v{s}),
	\end{equation}
	where $s_i \in [0,\infty)$ for $i=1,...,n_x$ and $\mat{R}$ is the same as in \eqref{eq:R_choice}. To utilize the results from \cite{gruber} we over-approximate the ellipsoid $\T$ by a zonotope $\widehat{\T}$, keeping the same center $\v{c}_{\widehat{\T}} = \v{c}_{\T}$, but constructing a new generator matrix
	\begin{equation}
		\mat{G}_{\widehat{\T}} = \mat{G}_{\T} \mat{G}_{\mathrm{fixed}},
	\end{equation}
	where $\mat{G}_{\mathrm{fixed}} \in \reals^{n_x\times m}$ has to satisfy $\ball_2 \subseteq Z(\mat{G}_{\mathrm{fixed}})$, which can by achieved by means of \cite[Theorem 4]{victor_ellipsoid} for $m>n_x$, and by taking $\mat{G}_{\mathrm{fixed}} = \mat{I}$ when $m = n_x$.
	We denote by $\mat{G}_{\widehat{\T}}$ the generator matrix of $\widehat{\T}$, and by $\v{c}_{\widehat{\T}}$ its center. If $\T$ is modeled as a zonotope, we set $\widehat{\T} = \T$.
	
	Whether $\T$ is a zonotope or an ellipsoid, for any $\v{x}_0 \in \T$, there exists a parameter $\v{\beta}(\v{x}_0) \in \ball_q$ (where $q = \infty$ if $\T$ is a zonotope, $q=2$ if $\T$ is an ellipsoid) such that
	\begin{equation}
		\label{eq:x_0_param}
		\v{x}_0 = \mat{G}_{\T}\v{\beta}(\v{x}_0) + \v{c}_{\T}.
	\end{equation}
	
	\subsection{Computation of the Safe Set}
	\label{sec:reachability}
	
	We now present how to model the constraint \eqref{eq:terminalRegion}, so that it can be solved using a standard optimization solver. For the constraint ${\R_{\v{x}}(\tend; \widehat{\T}) \subseteq \T}$, we assume $\R_{\v{x}}(\tend; \widehat{\T})$ can be over-approximated using the zonotope $Z(\mat{G}_{\Nts}, \v{c}_{\Nts})$, and $\T = E_q(\mat{H}\Diag(\v{s}), \v{c}_{\T})$ with $\mat{H} = \mat{R}\,\mat{G}_{\mathrm{fixed}}$ for $q=\infty$ and $\mat{H} = \mat{R}$ for $q=2$. By Theorem \ref{thm:main_thm}, ${\R_{\v{x}}(\tend; \widehat{\T}) \subseteq \T}$ holds if there exists a matrix $\mat{Z}$ such that
	\begin{equation}
		\norm{\mat{Z}}_{L_{1,q}^\top} \leq 1, \quad \mat{H}\Diag(\v{s})\mat{Z} = \begin{bmatrix}\mat{G}_{\Nts} & \v{c}_{\Nts}-\v{c}_{\T} \end{bmatrix}.
	\end{equation}
	With $s_j > 0$ for every $j=1,...,m$, we can make the variable transformation $\mat{Z}' = \Diag(\v{s})\mat{Z}$, which yields
	\begin{equation}
		\begin{split}
			&\norm{\Diag(\v{s})^{-1}\mat{Z}'}_{L_{1,q}^\top} \leq 1,\\
			&\mat{H}\Diag(\v{s})\mat{Z}' = \begin{bmatrix}\mat{G}_{\Nts} & \v{c}_{\Nts}-\v{c}_{\T} \end{bmatrix}.
		\end{split}
	\end{equation}
	Let $\v{\zeta}_j'$ be the $j$-th row of $\mat{Z}'$. For $q=\infty$, the constraint $\norm{\Diag(\v{s})^{-1}\mat{Z}'}_{L_{1,q}^\top} \leq 1$ is equivalent to
	\begin{equation}
		\nnorm{\v{\zeta}_j'}_1 \leq s_j, \quad \forall j \in \{1,...,m\},
	\end{equation}
	which can be expressed as a system of linear constraints. As for $q = 2$, the constraint $\norm{\Diag(\v{s})^{-1}\mat{Z}'}_{L_{1,q}^\top} \leq 1$ is equivalent to
	\begin{equation}
		\label{eq:ell_weaker_containment}
		\sum_j \frac{\nnorm{\v{\zeta}_j}^2_1}{s_j^2} \leq 1,
	\end{equation}
	which is not a convex constraint. However, the stronger assumption
	\begin{equation}
		\label{eq:ell_stronger_containment}
		\nnorm{\v{\zeta}_j}_1 \leq \frac{s_j}{\sqrt{m}}, \quad \forall j\in\{1,...,m\}
	\end{equation}
	can be represented as a system of linear constraints. Since \eqref{eq:ell_stronger_containment} implies \eqref{eq:ell_weaker_containment}, we can use \eqref{eq:ell_stronger_containment} as a sufficient (but not necessary) criterion for containment instead.
	
	The final optimization problem we have to solve has the form
	\begin{equation}
		\label{eq:main_opt_problem}
		\begin{split}
			&\max_{\substack{\v{c}_{\T}, s_j \geq 0\\\mathstrut\mat{Z}, \mat{U}_{i}, \v{c}_{u,i}}} \sqrt[m]{s_1\cdot...\cdot s_m}\\
			&\text{subject to}\\
			&\R_{\v{x}}(t; \widehat{\T}) \subseteq \X, \quad \R_{\v{u}}(t; \widehat{\T}) \subseteq \U\\
			&\mat{H}\,\mat{Z} = \begin{bmatrix}\mat{G}_{\Nts} & \v{c}_{\Nts}-\v{c}_{\T} \end{bmatrix},\\
			&\nnorm{\v{\zeta}_j}_1 \leq \frac{s_j}{\sqrt[q]{m}},
		\end{split}
	\end{equation}
	where $\v{\zeta}_j$ are the rows of $\mat{Z}$ and for $q = \infty$ we set $\sqrt[q]{m} = 1$. The cost function $\sqrt[m]{s_1\cdot...\cdot s_m}$, suggested in \cite{gruber}, approximates the volume of $\T$ and is concave, such that the optimization problem may be formulated as a convex conic optimization problem in the variables $\v{c}_{\T}$, $s_j$, $\mat{Z}$, $\mat{U}_{i}$, and $\v{c}_{u,i}$, which together have a combined degree of freedom of $n_x+m+l(m_{\Nts}+1)+n_um_{u}(\Nts+1)+n_u(\Nts+1)$, where $l$ is the number of columns of $\mat{H}$ (either $m$ if $q=\infty$ or $n_x$ if $q = 2$) and $m_{\Nts}$ is the number of columns of $\mat{G}_{\Nts}$. By induction, one can show $m_{\Nts} = m + \Nts n_w (\eta+1) n_x$, where $\eta$ is a parameter of the algorithm, known as the Taylor order in \cite{althoff_dissertation}. Therefore, solving \eqref{eq:main_opt_problem} requires $\mathcal{O}(M^{3.5}\ln(M))$ operations with $M := l(m + \Nts n_wn_x\eta)+n_um_{u}\Nts$, according to Remark \ref{rmk:optimization_complexity}.
	When deploying the controller, it remains to compute $\v{u}(t)$ for a given initial point $\v{x}_0$ and time $t$. As discussed in \eqref{eq:x_0_param}, this can be done by computing a solution $\v{\beta} \in \ball_q$ of
	\begin{equation}
		\v{x}_0 = \v{c}_{\widehat{\T}} + \mat{G}_{\T}\v{\beta},
	\end{equation}
	which can be solved using linear programming in the case $q = \infty$. If $q=2$, by construction $\mat{G}_{\T}$ is necessarily invertible, thus $\v{\beta}$ can be determined without solving an optimization problem:
	\begin{equation}
		\v{\beta} = \mat{G}_{\T}^{-1}\left(\v{x}_0 - \v{c}_{\widehat{\T}}\right).
	\end{equation}
	Note that this has a significantly faster runtime than the approach for $q=\infty$ that requires solving a linear program.
	
	\subsection{Numerical Results for Platoon Benchmark}
	We apply our methods to the $2k$-dimensional platooning benchmark with $k$ vehicles. This model is described in \cite[Section IV.]{schuermann2017} for $k=4$, but can easily be generalized to arbitrary $k\in\naturals$. The model is explicitly given by the system
	\begin{equation}
		\label{eq:platoon_4_vehicles}
		\begin{split}
			&\dot{x}_1 = x_2,\\
			&\dot{x}_{2i-1} = x_{2i},\\
			&
		\end{split}
		\quad
		\begin{split}
			&\dot{x}_2 = u_1 + w_1,\\
			&\dot{x}_{2i} = u_{i-1} - u_i + w_{i-1} - w_i \text{ for } i=2,...,k,\\
			&
		\end{split}
	\end{equation} 
	where $x_1$ and $x_2$ are the position and velocity of the leading vehicle, and $x_{2i-1}$ and $x_{2i}$ for $i\geq 2$ are the position and velocity of the $i$-th vehicle relative to the previous vehicle. We assume $\U = [\SI{-10}{\metre\per\second\tothe{2}},\SI{10}{\meter\per\second\tothe{2}}]^k$, $\W = [\SI{-1}{\meter\per\second\tothe{2}},\SI{1}{\meter\per\second\tothe{2}}]^k$, and constrain the state such that the vehicles do not collide, i.e., $x_{2i-1}\geq 0$ for $i\geq 2$. Zonotopes and ellipsoids are implemented using the CORA toolbox \cite{cora}. In addition, we solve convex programs using the MOSEK optimization suite \cite{mosek}, in combination with the YALMIP toolbox \cite{Lofberg2004}. We also compare our methods to the approach from \cite{gruber} implemented in the AROC toolbox \cite{aroc}. All computations were made in MATLAB on an Intel Core i7-8650U CPU @1.9GHz with 24GB memory. The scripts of our results, including the code generating the figures in this document, are available at
	
	\medskip
	
	\begin{center}
		\url{https://codeocean.com/capsule/5598835/tree}
	\end{center}
	
	\medskip
	
	The algorithms presented in Section \ref{sec:containment_problem} will be made available for the next CORA release, and the algorithms from Section \ref{sec:applications} will be accessible for the next AROC release.
	
	For $k\in\{2,4,6,8,10\}$, we used $\Nts=10\cdot(k+1)$ time steps with time-step size $\SI{0.1}{s}$, and chose $m=5k$ fixed generators for zonotopes, and $m=2k$ for ellipsoids. We compared the runtime of our algorithms to that of \cite{gruber}, which for the remainder of this section we refer to as the terminal set approach. As a heuristic for the overall size of the safe set, we computed the volume of the safe set (for $k=10$ and when the safe set was a zonotope, we did not compute the volume, as this would have required several days of computation). To compare how fast each approach computes the controller, we simulated 100 random trajectories starting at random points on the boundary of $\T$, while evaluating the controller at each time step for random disturbances sampled uniformly from $\W$. All results are displayed in Tab.~\ref{table:runtimes}-\ref{table:volume}. We also provide a graphical representation of the projection of the safe set on the first 2 dimensions for $k = 4$ in Fig.~\ref{fig:large_platoon}.
	
	Both our approaches yield significantly better results than \cite{gruber} with respect to the size of the safe set, which is due to the fact that the algorithm in \cite{gruber} returns a set of states that can be steered into a fixed terminal set. 
	This is more conservative than our methods, which have a variable initial and terminal set. In addition, both our algorithms were faster by several orders of magnitude, both in terms of total runtime and the computation time of the controller; in fact, the algorithm from \cite{gruber} would have taken several days of computation for $k\geq 6$, which is why we decided not to compute the results in those cases. 
	
	Comparing our zonotope and ellipsoid approaches, we see zonotopes yield a larger safe set, though in practice representing terminal regions in model predictive control using ellipsoids is beneficial, since the terminal constraint $\v{x} \in \T$ can be formulated using a single second-order cone constraint, whereas using the point-in-zonotope containment encoding from \cite{Kulmburg2021} requires $\mathcal{O}(n_x^2)$ auxiliary variables, and converting the zonotope into an H-polytope introduces up to $2\binom{m}{n_x}$ linear constraints \cite[Section 5.1.2.]{ALTHOFF2010233}. Furthermore, the computation of the controller is faster by a factor of about 25\% for ellipsoids, which suggests ellipsoids would be better suited for real-time applications.
	\begin{figure}[t!]
		\centering
		\definecolor{mycolor1}{rgb}{0.00000,0.36078,0.67059}%
		\definecolor{mycolor2}{rgb}{0.89020,0.10588,0.13725}%
		\definecolor{mycolor3}{rgb}{1.00000,0.76471,0.14510}%
%
		\caption{Projection (onto the first two coordinates) of the terminal region for platoon benchmark with four vehicles. Both our approaches yield larger sets than \cite{gruber}.}
		\label{fig:large_platoon}
	\end{figure}
	
	\begin{table}[t!]
		\centering
		\caption{Runtime for platooning benchmark with $k$ vehicles}
		\begin{tabular}{crrr} 
			\toprule
			\multirow{3}{*}{$k$}	& \multicolumn{3}{c}{Runtime in \SI{}{\second}}\\\cmidrule(lr){2-4}
			& \multirow{2}{*}{Ellipsoid}	& \multirow{2}{*}{Zonotope} 	& \multicolumn{1}{c}{Terminal set}\\
			&								&								& \multicolumn{1}{c}{approach from \cite{gruber}}\\\midrule
			2   & \textbf{3}			& 4			& 64			\\
			4	& \textbf{15}		& 19		& 21875		\\
			6	& \textbf{77}		& 103		& 	-		\\
			8	& \textbf{381}		& 425		& 	-		\\
			10	& \textbf{1066}		& 1825		& 	-		\\\bottomrule
		\end{tabular}
		\label{table:runtimes}
		
		\vspace{0.5cm}
		
		\centering
		\caption{Controller evaluation for platooning benchmark with $k$ vehicles, averaged over 100 runs}
		\begin{tabular}{crrr} 
			\toprule
			\multirow{3}{*}{$k$}	& \multicolumn{3}{c}{Runtime in \SI{}{\milli\second}}\\\cmidrule(lr){2-4}
			& \multirow{2}{*}{Ellipsoid}	& \multirow{2}{*}{Zonotope} 	& \multicolumn{1}{c}{Terminal set}\\
			&								&								& \multicolumn{1}{c}{approach from \cite{gruber}}\\\midrule
			2   & \textbf{0.3$\pm$0.1}			& 3.6$\pm$0.1			& 3.0$\pm$0.1			\\
			4	& \textbf{0.5$\pm$0.1}		& 4.0$\pm$0.1		& 5.0$\pm$2.8		\\
			6	& \textbf{0.9$\pm$0.1}		& 5.4$\pm$0.1		& 	-		\\
			8	& \textbf{1.0$\pm$0.1}		& 5.2$\pm$0.1		& 	-		\\
			10	& \textbf{1.2$\pm$0.1}		& 5.9$\pm$0.1		& 	-		\\\bottomrule
		\end{tabular}
		\label{table:controller}
		
		\vspace{0.5cm}

		\centering
		\caption{Volume of safe set of platooning benchmark with $k$ vehicles}
		\begin{tabular}{crrr} 
			\toprule
			\multirow{3}{*}{$k$}	& \multicolumn{3}{c}{Volume in $\SI{}{\meter}^{2k}$/$\SI{}{\second}^k$}\\\cmidrule(lr){2-4}
			& \multirow{2}{*}{Ellipsoid}	& \multirow{2}{*}{Zonotope} 	& \multicolumn{1}{c}{Terminal set}\\
			&								&								& \multicolumn{1}{c}{approach from \cite{gruber}}\\\midrule
			2   & 4.6$\times 10^{5}$		& \bfseries 3.6$\bm{\times 10^{6}}$			& 1.0$\times 10^{4}$			\\
			4	& 6.5$\times 10^{9}$		& \bfseries 6.6$\bm{\times 10^{12}}$			& 2.0$\times 10^{6}$		\\
			6	& 5.3$\times 10^{13}$		& \bfseries 6.9$\bm{\times 10^{18}}$			& 	-		\\
			8	& 2.5$\times 10^{16}$		& \bfseries 2.5$\bm{\times 10^{25}}$			& 	-		\\
			10	& \bfseries 4.9$\bm{\times 10^{18}}$& -		& 	-		\\\bottomrule
		\end{tabular}
		\label{table:volume}
	\end{table}
	
	\section{Conclusion}
	We have explored three distinct aspects of the ellipsotope containment problem: First, we established a link between the containment problem and the computation of generalized matrix norms, which allowed us to formulate strategies to solve, or at least approximate, the containment problem based on already known methods for matrix norms. We then exploited this link to determine the computational complexity of the containment problem, showing that in most instances the problem is at least $\tau$-inapproximable. Finally, we demonstrated how some of the approximative methods deduced in Section \ref{sec:containment_problem} can be used to compute safe terminal regions for dual mode model predictive control in Section \ref{sec:applications}.
	
	There remain several open questions about the containment problem for ellipsotopes, for instance concerning the cases $1<p<2<q<\infty$, where nothing is known about the approximability. Even for certain cases where inapproximability results are known (e.g., the zonotope-in-zonotope containment problem, when $p=q=\infty$), current algorithms seem to perform better than expected. On the other hand, we have only applied the containment problem to linear time-invariant systems. It would be interesting to apply similar methods to systems with an output feedback loop, or even to nonlinear systems.
	
	\appendix
	
	\subsection{Proof of Theorem \ref{thm:not_FPTAS}}
	\label{sec:hardness_ell_in_zono}
	Our aim is to formulate the problem of computing $\norm{\mat{A}}_{p\mapsto 1}$ as a special instance of the ellipsotope-in-zonotope containment problem, thus showing that the containment problem is at least as hard to solve as calculating the $p\mapsto1$-norm of a matrix. The key idea consists in constructing a continuous function $\sigma(\rho)$ which satisfies $\sigma(\rho) = 1$ if and only if $\rho = \norm{\mat{A}}_{p\mapsto 1}$ for some matrix $\mat{A}$. Evaluating $\sigma(\rho)$ will be shown to be equivalent to a containment problem, so that we can approximate $\norm{\mat{A}}_{p\mapsto 1}$ by searching for a root of $\sigma(\rho) - 1$. The special instance of the containment problem we will use (of which a graphical representation can be found in Fig.~\ref{fig:gritzmann_kulmburg_construct}) is constructed as follows:
	For $n \in \naturals$ with $n> 1$, let
	\begin{equation}
		\label{eq:apx_H_def}
		\begin{split}
			&\mat{H} := \begin{pmatrix}
				\mat{I}_n & -\mat{I}_n\\
				\v{1}_n^{\,\top} & \v{1}_n^{\,\top}
			\end{pmatrix} \in \reals^{(n+1) \times 2n},
		\end{split}
	\end{equation}
	and define the circumbody $\Zcirc$ as
	\begin{equation}
		\Zcirc := \Set{\mat{H}\v{\beta}}{\nnorm{\v{\beta}}_{\infty}\leq 1, \v{\beta}\geq \v{0}},
	\end{equation}
	where $\v{\beta} \geq \v{0}$ means $\beta_i \geq 0$ for all $i=1,...,2n$. For technical reasons, we also consider the set
	\begin{align*}
		&\ball_1 \times [1,2n-1] =\\
		&\qquad\Set{\v{x}\in\reals^{n+1}}{\sum_{i=1}^n|x_i| \leq 1 \text{ and } x_{n+1}\in [1,2n-1]},
	\end{align*}
	which is a tube of height $2n-2$, with $\ball_1$ being a cross-section. The set $\Zcirc$ is indeed a zonotope, containing $\ball_1 \times [1,2n-1]$ in its middle part (see Fig.~\ref{fig:gritzmann_kulmburg_construct}), and slices of its bottom part are scaled versions of the cross-polytope $\ball_1$: 
	\begin{lemma}
		\label{lmm:classical_zono_cut}
		The set $\Zcirc$ is a zonotope with $\Zcirc = Z(\frac{1}{2}\mat{H}, n\v{e}_{n+1})$, and for any $\rho \in [0,1]$
		\begin{equation}
			\label{eq:cut_is_cross_polytope}
			\Set{\v{x}\in\reals^{n+1}}{x_{n+1} = \rho} \cap \Zcirc = (\rho \, \ball_1) \times \{\rho\}.
		\end{equation}
		Furthermore, $\ball_1 \times [1,2n-1] \subset \Zcirc$.
	\end{lemma}
	\begin{proof}
		We proved $\Zcirc = Z(\frac{1}{2}\mat{H}, n\v{e}_{n+1})$ in \cite{Kulmburg2021} in the discussion after equation (33), and \eqref{eq:cut_is_cross_polytope} was proven in \cite[Lemma 5.]{Kulmburg2021}, so it remains to prove $\ball_1 \times [1,2n-1] \subset \Zcirc$. The set $\ball_1 \times [1,2n-1]$ is clearly a polytope with vertices $\v{v}_{1, i, \pm} := \begin{bmatrix}\pm\v{e}_i^\top & 1\end{bmatrix}^\top$ and $\v{v}_{2n-1, i, \pm} := \begin{bmatrix}\pm\v{e}_i^\top & 2n-1\end{bmatrix}^\top$ for $i=1,...,n$. By convexity, it suffices to prove that all these vertices are contained in $\Zcirc$. We define the following vectors in $\reals^{2n}$ for $i=1,...,n$:
		\begin{align*}
			&\v{\beta}_{1, i, +} := -\v{1} + 2\v{e}_i, \quad &\v{\beta}_{1, i, -} := -\v{1} + 2\v{e}_{n+i},\\
			&\v{\beta}_{2n-1, i, +} := \v{1} - 2\v{e}_{n+i}, \quad & \v{\beta}_{1,i,-} := \v{1} - 2\v{e}_i.
		\end{align*}
		One can verify that these vectors satisfy the relations
		\begin{align*}
			&\frac{1}{2}\mat{H}\v{\beta}_{1, i, \pm} + n\v{e}_{n+1} = \v{v}_{1, i, \pm},\\
			&\frac{1}{2}\mat{H}\v{\beta}_{2n-1, i, \pm} + n\v{e}_{n+1} = \v{v}_{2n-1, i, \pm},\\
			&\nnorm{\v{\beta}_{1, i, \pm}}_{\infty} \leq 1, \; \nnorm{\v{\beta}_{2n-1, i, \pm}}_{\infty} \leq 1,
		\end{align*}
		which shows $\v{v}_{1, i, \pm}, \v{v}_{2n-1, i, \pm} \in Z(\frac{1}{2}\mat{H}, n\v{e}_{n+1}) = \Zcirc$.	
	\end{proof}
	
	For the inbody, we use the ellipsotope $\Ein = E_p(\mat{G}_{\mat{A}}, n\v{e}_{n+1})$, where
	\begin{equation}
		\label{eq:apx_G_def}
		\mat{G}_{\mat{A}} = \begin{pmatrix}
			\mat{A} & \v{0} \\
			\v{0}^{\,\top}& -L
		\end{pmatrix},
	\end{equation}
	with $\mat{A} \in \reals^{n\times m}$ being an arbitrary matrix. If $L$ is chosen appropriately, $\Ein$ touches the boundary of $\Zcirc$ in the top and bottom part of $\Zcirc$ if and only if $\norm{\mat{A}}_{p\mapsto 1} = 1$. The ideal way to choose $L$ is described in the following Lemma; the proof that it satisfies the aforementioned properties will be investigated later in Lemma \ref{lmm:main_technical_lemma_apx}:
	\begin{lemma}
		\label{lmm:L_properties}
		For $n\in\naturals\backslash\{1\}$ and $p\in(1,\infty)$, there exist $L \in (n-1,n)$ and $\varrho \in (0,1)$ satisfying the equations
		\begin{align}
			&1- \varrho^p = \left(\frac{n-\varrho}{L}\right)^p,\label{eq:varrho_L_1}\\
			&L = \left(1-\left(\frac{n-\varrho}{L}\right)^p\right)^{\frac{1}{p}-1}\left(\frac{n-\varrho}{L}\right)^{p-1}.\label{eq:varrho_L_2}
		\end{align}
		For $p = \infty$, we set $L = n-1$ and $\varrho = 1$.
	\end{lemma}
	\begin{proof}
		We set $\varrho := n^{-1/(p-1)}$ and
		\begin{equation}
			\label{eq:alt_form_L}
			L:= \frac{n-\varrho}{(1-\varrho^{p})^{1/p}}.
		\end{equation}
		Assuming we establish $L > 0$, the definition of $L$ is directly equivalent to \eqref{eq:varrho_L_1}, and we can verify that \eqref{eq:varrho_L_2} is valid by plugging in the definition of $L$ and $\varrho$. The fact that $\varrho \in (0,1)$ directly follows from the facts that $n>1$. Thus, it remains to show $L \in (n-1,n)$. Using the explicit formula of $\varrho$ in the formula of $L$, we see that
		\begin{equation*}
			L = n\cdot (1-n^{-1-\frac{1}{p-1}})^{1-1/p}.
		\end{equation*}
		The right term is clearly smaller than $1$ since $n > 1$ and $p \in (1,\infty)$, which establishes $L \leq n$, so we only need to show $L \geq n-1$. This can easily be seen to be equivalent to
		\begin{align*}
			&n-n^{-\frac{1}{p-1}} \geq (n-1)(1-n^{-\frac{p}{p-1}})^{1/p}\\
			\Leftrightarrow& n - n^{-\frac{1}{p-1}} \geq n(1-n^{-\frac{p}{p-1}})^{1/p} - (1-n^{-\frac{p}{p-1}})^{1/p}.
		\end{align*}
		Therefore, it would suffice to prove the two separate inequalities
		\begin{align}
			n \geq n(1-n^{-\frac{p}{p-1}})^{1/p},\label{eq:first_separate_inequality}\\
			n^{-\frac{1}{p-1}} \leq (1-n^{-\frac{p}{p-1}})^{1/p}.\label{eq:second_seprate_inequality}
		\end{align}
		Since $n > 1$, \eqref{eq:first_separate_inequality} is trivial, and \eqref{eq:second_seprate_inequality} is equivalent to
		\begin{align*}
			& n^{-\frac{p}{p-1}} \leq 1-n^{-\frac{p}{p-1}}\\
			\Leftrightarrow & 2 \cdot \frac{1}{2^{p}} \leq n.
		\end{align*}
		But since the stronger condition $n \geq 2$ holds, we conclude that \eqref{eq:second_seprate_inequality} also holds, and thus, by extension, $L \geq n-1$.
	\end{proof}
	\begin{figure}[b!]
		\centering
		\begin{tikzpicture}[scale=1]
			\draw [->,thick] (0,0) -- (0,4.5);
			\draw [->,thick] (-3.5,0) -- (3.5,0) node (xaxis) [right] {$\mathbb{R}^n$};
			
			\fill [black] (0,0) circle (2pt);   
			\draw (0,0) node (origin) [below] {$\vec{0}_{n+1}$};
			
			\draw [Blue] (0,0) -- (1,1) -- (2,2) -- (1,3) -- (0,4) -- (-1,3) -- (-2,2) -- (-1,1) -- (0,0);
			
			\draw [Blue] (2.4,3.1) node {$\Zcirc$};
			\draw [solid] (1.5, 2.5) -- (2.2,3);
			
			\draw [dashed, YellowDark] (-1,1) -- (1,1) -- (1,3) -- (-1,3) -- (-1,1);
			\draw [solid] (0.2, 3) -- (0.9,4.2);
			\draw (0.5, 4.5) node [right, YellowDark] {$\ball_1 \times[1, (2n-1)]$};
			
			\draw[color=Red] (0,2) ellipse (1 and 1.7348);
			\draw [solid] (-1.5,4) -- (-0.9,0.7562+2);
			\draw (-1.8, 4.3) node [right, Red] {$\Ein$};
			
			\fill [black] (0,2) circle (2pt);
			\draw (0,2) node [above right]{$\vec{c}$};
			
			\draw [dotted] (-3.5,1) -- (-1,1);
			\draw [dotted] (-3.5,3) -- (-1,3);
			
			\draw (-2.8,0.5) node {$S_{\text{bottom}}$};
			\draw (-2.8,2) node {$S_{\text{middle}}$};
			\draw (-2.8,3.5) node {$S_{\text{top\phantom{tom}}}$};
			
			\draw [dotted] (0,2-1.7348) -- (2.5,2-1.7348);
			\draw [dotted] (0,2) -- (2.5,2);
			
			\draw [<->,solid] (2.5,2-1.7348) -- (2.5,2);
			\draw [<->,solid] (2.5,0) -- (2.5,2-1.7348);
			
			\draw (2.8,1.2) node {$L$};
			\draw (2.8,0.16) node {$\varrho$};

		\end{tikzpicture}
		\caption{A graphical representation of $\Ein$ and $\Zcirc$ used for the proof of Theorem \ref{thm:not_FPTAS}. For $\v{z}\in\reals^{n+1}$, the horizontal axis represents the first $n$ coordinates of $\v{z}$, whereas the vertical axis represents the $(n+1)$-th entry of $\v{z}$. The vector $\v{c} = n\v{e}_{n+1}$ is the common center of $\Ein$ and $\Zcirc$. In this construction, we assume $\norm{\mat{A}}_{p\mapsto 1} = 1$, in which case the middle part of $\Ein$ fits tightly into the tube $\ball_1\times[1,(2n-1)]$.}
		\label{fig:gritzmann_kulmburg_construct}
	\end{figure}
	We now show that with this choice of $L$ (and $\varrho$), $\Ein$ touches the boundary of $\Zcirc$ if and only if $\norm{\mat{A}}_{p\mapsto 1} = 1$, and $\norm{\mat{A}}_{p\mapsto 1} > 1$ entails that $\Ein$ can not be contained in $\Zcirc$, whereas $\norm{\mat{A}}_{p\mapsto 1} < 1$ means that $\Ein$ is strictly contained in $\Zcirc$, without touching its boundary:
	\begin{lemma}
		\label{lmm:main_technical_lemma_apx}
		Let $p\in(1,\infty]$, $\mat{A} \in \reals^{n\times m}$, and suppose $\mat{G}_{\mat{A}}$ and $\mat{H}$ are chosen as in \eqref{eq:apx_G_def} and \eqref{eq:apx_H_def}, respectively. Then,
		\begin{equation}
			\label{eq:r_lemma_eq}
			r(E_p(\mat{G}_{\mat{A}}), Z(\mat{H})) = 1 \text{ if and only if } \norm{\mat{A}}_{p\mapsto 1} = 1.
		\end{equation}
		This remains true if the equalities in \eqref{eq:r_lemma_eq} are replaced by $<$ or $>$.
	\end{lemma}
	\begin{proof}
		Since
		\begin{equation}
			r(E_p(\mat{G}_{\mat{A}}), Z(\mat{H})) = r(E_p(\mat{G}_{\mat{A}}, n\v{e}_{n+1}), Z(\mat{H}, n\v{e}_{n+1})),
		\end{equation}
		we focus on the containment $E_p(\mat{G}_{\mat{A}}, n\v{e}_{n+1}) \subseteq Z(\mat{H}, n\v{e}_{n+1})$, as it makes key arguments easier to follow. Our first goal is to show that if $\norm{\mat{A}}_{p\mapsto 1} = 1$, then $E_p(\mat{G}_{\mat{A}}, n\v{e}_{n+1}) \subseteq Z(\mat{H}, n\v{e}_{n+1})$, and that this containment is tight in the sense that some point in $E_p(\mat{G}_{\mat{A}}, n\v{e}_{n+1})$ lies on the boundary of $Z(\mat{H}, n\v{e}_{n+1})$. 	
		For this purpose, we define the three regions $S_{\mathrm{bottom}} = \reals^n\times [0,1]$, $S_{\mathrm{middle}} = \reals^n\times[1,2n-1]$, and $S_{\mathrm{top}} = \reals^n\times[2n-1, 2n]$ (see also Fig. \ref{fig:gritzmann_kulmburg_construct}). Clearly, $Z(\mat{H}, n\v{e}_{n+1})$ is contained in the union of these three sets, and $E_p(\mat{G}_{\mat{A}}, n\v{e}_{n+1}) \subseteq Z(\mat{H}, n\v{e}_{n+1})$ holds if and only if $(S \cap E_p(\mat{G}_{\mat{A}}, n\v{e}_{n+1})) \subseteq Z(\mat{H}, n\v{e}_{n+1})$ for $S = S_{\mathrm{bottom}}, S_{\mathrm{middle}},$ and $S_{\mathrm{top}}$. Since $E_p(\mat{G}_{\mat{A}}, n\v{e}_{n+1})$ is centrally symmetric around $n\v{e}_{n+1}$, it suffices to analyze containment for $S_{\mathrm{bottom}}$ and $S_{\mathrm{middle}}$ (i.e., the analysis for $S_{\mathrm{bottom}}$ and $S_{\mathrm{top}}$ is identical).
		
		We begin with the middle part for $p \neq \infty$: Let $\rho \in [0,2n]$, and consider the slice $\Set{\v{x}\in\reals^{n+1}}{x_{n+1} = \rho} \cap E_p(\mat{G}_{\mat{A}}, n\v{e}_{n+1})$, which is either empty (this can only happen if $\rho > n+L$ or $\rho < n - L$) or consists of the points
		\begin{align*}
			\v{x} = \begin{bmatrix}
				\mat{A}\v{\alpha}'\\
				-L\alpha_{m+1} + n
			\end{bmatrix},&\text{ such that } \left(\norm{\v{\alpha}'}_p^p + |\alpha_{m+1}|^p\right)^{1/p} \leq 1,\\
			& \quad \text{and } -L\alpha_{m+1} + n = \rho.
		\end{align*}
		This is equivalent to
		\begin{equation*}
			\v{x} = \begin{bmatrix}
				\mat{A}\v{\alpha}'\\
				\rho
			\end{bmatrix}, \text{ such that } \norm{\v{\alpha}'}_p \leq \left(1 - \left|\frac{n-\rho}{L}\right|^p\right)^{1/p}.
		\end{equation*}
		We define the radius $\theta(\rho)$ of this slice with respect to the 1-norm as $\theta(\rho) = 0$ for $\rho \not\in [n-L,n+L]$, and for $\rho \in [n-L,n+L]$
		\begin{equation}
			\label{eq:theta_def}
			\begin{split}
				\theta(\rho) &= \max_{\v{\alpha}'} \Set{\norm{\mat{A}\v{\alpha}'}_1}{\norm{\v{\alpha}'}_p \leq \left(1 - \left|\frac{n-\rho}{L}\right|^p\right)^{1/p}}\\
				&= \left(1 - \left|\frac{n-\rho}{L}\right|^p\right)^{1/p} \cdot \norm{\mat{A}}_{p\mapsto 1}\\
				&=\left(1 - \left|\frac{n-\rho}{L}\right|^p\right)^{1/p},
			\end{split}
		\end{equation}
		where the second equality follows from the absolute homogeneity of norms and \cite[Proposition III.2.1.]{conway_course_2007}, and for the last equality we used $\norm{\mat{A}}_{p\mapsto 1} = 1$.
		The maximum of $\theta(\rho)$ occurs at $\rho = n$, where $\theta(n) = 1$. By Lemma \ref{lmm:classical_zono_cut}, we conclude $(S_{\mathrm{middle}} \cap E_p(\mat{G}_{\mat{A}}, n\v{e}_{n+1})) \subseteq \ball_1 \times [1,2n-1] \subseteq Z(\mat{H}, n\v{e}_{n+1})$.
		If $p = \infty$, the radius of $E_p(\mat{G}_{\mat{A}}, n\v{e}_{n+1})$ at height $\rho$, with respect to the 1-norm, is
		\begin{equation}
			\theta(\rho) =
			\begin{cases}
				1, \text{ if } \rho\in [1,2n-1],\\
				0, \text{ otherwise}.
			\end{cases}
		\end{equation}
		For the same reasons as for the case $p \in (1,\infty)$, this proves $(S_{\mathrm{middle}} \cap E_{\infty}(\mat{G}_{\mat{A}}, n\v{e}_{n+1})) \subseteq  Z(\mat{H}, n\v{e}_{n+1})$.
		It remains to prove containment for the bottom part, which corresponds to the case $\rho \in [0, 1]$. We define the function
		\begin{equation}
			f(\rho) = \rho - \theta(\rho),
		\end{equation}
		which is non-negative if and only if the slice $\Set{\v{x}\in\reals^{n+1}}{x_{n+1} = \rho} \cap E_p(\mat{G}_{\mat{A}}, n\v{e}_{n+1})$ is contained in $(\rho \ball_1) \times \{\rho\}$. By Lemma \ref{lmm:classical_zono_cut}, $(S_{\mathrm{bottom}} \cap E_p(\mat{G}_{\mat{A}}, n\v{e}_{n+1})) \subseteq Z(\mat{H}, n\v{e}_{n+1})$ if and only if $f(\rho) \geq 0$ for all $\rho \in [0, 1]$, and $f(\rho) = 0$ for $\rho > 0$ implies the existence of a point in $E_p(\mat{G}_{\mat{A}}, n\v{e}_{n+1})$ that touches the boundary of $Z(\mat{H}, n\v{e}_{n+1})$ at height $\rho$. For $p = \infty$, $f(1) = 0$ and $f(\rho) = \rho \geq 0$ for $\rho \in [0, 1)$, so we are left with the case $p\in (1, \infty)$. By construction, $f(0) = 0$ and $f(\rho) = \rho > 0$ for $\rho \in (0,n-L]$, so it remains to analyze the range $(n-L, 1]$, on which $\theta$ (and thus $f$) is smooth in $\rho$. Since $n-1 < L < n$ by Lemma \ref{lmm:L_properties}, for $\rho \in (n-L, 1]$
		\begin{equation}
			\label{eq:quick_inquality_Lnrhop}
			\left(\frac{n-\rho}{L}\right)^p \in (0, 1).
		\end{equation}
		Using \eqref{eq:quick_inquality_Lnrhop} and the fact that $n> 1$, one can verify
		\begin{equation*}
			\frac{d^2f(\rho)}{d\rho^2} > 0
		\end{equation*}
		for $\rho \in (n-L, 1]$. Thus, any extremal point of $f$ on $(n-L,1]$ is a minimum. The conditions
		\begin{equation*}
			\begin{split}
				f(\rho) = 0, \quad
				\frac{df(\rho)}{d\rho} = 0
			\end{split}
		\end{equation*}
		are equivalent to \eqref{eq:varrho_L_1} and $\eqref{eq:varrho_L_2}$, so that $0 = f(\varrho) \leq f(\rho)$ for any $\rho \in (n-L, 1]$ and, thus, $r(E_p(\mat{G}_{\mat{A}}), Z(\mat{H})) = 1$, which completes the proof that $\norm{\mat{A}}_{p\mapsto 1} = 1$ implies $r(\Ein, \Ecirc) = 1$. To prove the converse, suppose $1 \neq \norm{\mat{A}}_{p\mapsto 1} =: \xi$. Analogously to $\theta(\rho)$, we define the radius $\theta_{\xi}(\rho)$ (with respect to the 1-norm) of the slice of $\Ein$ at height $\rho$, which is now given for $p\in(1,\infty)$ as
		\begin{equation}
			\theta_{\xi} = \begin{cases}
				\xi\left(1 - \left|\frac{n-\rho}{L}\right|^p\right)^{1/p}, & \text{ if } \rho \in [n-L, n+L],\\
				0, & \text{ otherwise.}
			\end{cases}
		\end{equation}
		In other words, $\theta_{\xi}(\rho) = \xi \theta(\rho)$, which we analyze case by case:
		\begin{itemize}
			\item If $\xi = 0$, $E_p(\mat{G}_{\mat{A}}, n\v{e}_{n+1})$ is just the line segment from $(\v{0}^{\,\top},n-L)^\top$ to $(\v{0}^{\,\top},n+L)^\top$, which is contained in $Z(\mat{H}, n\v{e}_{n+1})$ but does not touch its boundary, thus $r(E_p(\mat{G}_{\mat{A}}), Z(\mat{H})) < 1$.
			\item If $0 < \xi < 1$, we have $\theta_{\xi}(\rho) < \theta(\rho)$ for every $\rho \in (n-L,n+L)$, and thus $E_p(\mat{G}_{\mat{A}}, n\v{e}_{n+1})$ is entirely contained in $E_p(\mat{G}_{\mat{A}/\xi}, n\v{e}_{n+1})$, only touching its boundary at the points $(\v{0}^{\,\top},n-L)^\top$ and $(\v{0}^{\,\top},n+L)$. Since these two points do not lie on the boundary of $Z(\mat{H}, n\v{e}_{n+1})$, the set $E_p(\mat{G}_{\mat{A}}, n\v{e}_{n+1})$ does not touch any boundary point of $Z(\mat{H}, n\v{e}_{n+1})$.
			\item If $1 < \xi$, then $\theta_{\xi}(\varrho) > \theta(\varrho) = \varrho$, so the slice $\Set{\v{x}\in\reals^{n+1}}{x_{n+1} = \rho} \cap E_p(\mat{G}_{\mat{A}}, n\v{e}_{n+1})$ can not be contained in $Z(\mat{H}, n\v{e}_{n+1})$, implying $r(E_p(\mat{G}_{\mat{A}}), Z(\mat{H})) > 1$.
		\end{itemize}
		In any case, $\Ein$ is either strictly contained or not contained in $\Ecirc$, which implies $r(E_p(\mat{G}_{\mat{A}}), Z(\mat{H})) \neq 1$. For $p=\infty$, the argument is almost identical, since then $\theta_{\xi}(\rho) = \xi$ if $\rho \in [1,2n-1]$, and $\theta_{\xi}(\rho) = 0$ otherwise.	
	\end{proof}
	
	We can now define the continuous function $\sigma(\xi)$ that will be used to determine $\norm{\mat{A}}_{p\mapsto 1}$ using a bisection algorithm:
	\begin{proposition}
		\label{prop:apx_main_prop}
		For $\mat{A} \in \reals^{n\times m}$ with $\mat{A} \neq \mat{0}$, let $\mat{H}$ and $\mat{G}_{\mat{A}}$ be defined as above.
		For $\xi \in (0, \infty)$, define
		\begin{equation}
			\label{eq:sigma_def}
			\sigma(\xi) := r(E_p(\mat{G}_{\frac{1}{\xi}\mat{A}}), Z(\mat{H})).
		\end{equation}
		Then $\sigma$ is continuous, $\sigma(\xi) \rightarrow \infty$ for $\xi\rightarrow 0$, and $\sigma(\widehat{\xi}) \leq 1$ for
		\begin{equation}
			\label{eq:widehat_xi_def}
			\widehat{\xi} = m^{\frac{p-1}{p}}\norm{\mat{A}}_{1\mapsto 1}.
		\end{equation}
		Additionally, $\sigma(\xi) = 1$ has the unique solution $\xi = \norm{\mat{A}}_{p\mapsto 1}$, and for any $\xi, \xi' \in (0, \widehat{\xi}\;]$
		\begin{equation}
			\label{eq:bound_sigma}
			\left|\sigma(\xi') - \sigma(\xi)\right| \geq \frac{\norm{\mat{A}}_{1\mapsto \infty}}{\widehat{\xi}^2}\left|\xi - \xi'\right|.
		\end{equation}
	\end{proposition}
	\begin{proof}
		We begin with the continuity of $\sigma(\xi)$: Since $\mat{H}$ has full rank, by \cite[p. 112, Theorem 2]{james_generalised_1978} the solutions $\v{\beta}$ to $\frac{1}{2}\mat{H} \v{\beta} = \frac{1}{\xi}\mat{G}_{\mat{A}}\v{\alpha}$ are given as
		\begin{equation*}
			\v{\beta} = 2\mat{H}^+\mat{G}_{\frac{1}{\xi}\mat{A}}\v{\alpha} + (I - \mat{H}^+\mat{H})\v{\omega},
		\end{equation*}
		where $\v{\omega} \in \reals^{2n}$ is arbitrary. We can thus rewrite $\sigma$ as
		\begin{equation}
			\label{eq:sigma_alternative}
			\sigma(\xi) = \max_{\norm{\v{\alpha}}_p\leq 1} \min_{\v{\omega}} \norm{\mat{H}^+\mat{G}_{\frac{1}{\xi}\mat{A}}\v{\alpha} + (I - \mat{H}^+\mat{H})\v{\omega}}_{\infty}.
		\end{equation}
		The cost function is continuous over $\v{\alpha}\in \reals^{m}$, $\v{\omega}\in\reals^{2n}$, and $\xi \in (0, \infty)$, thus $\sigma(\xi)$ is continuous over $\xi\in(0,\infty)$. We continue with the limit for $\xi\rightarrow 0$. We need an explicit formula for $\mat{H}^+$; we leave it to the reader to verify that
		\begin{equation}
			\mat{H}^+ =
			\frac{1}{2}\begin{bmatrix}
				\mat{I} & \frac{1}{n}\v{1}\\
				-\mat{I} & \frac{1}{n} \v{1}\\
			\end{bmatrix}
		\end{equation}
		satisfies all the requirements of a Moore-Penrose pseudo-inverse.	
		Since $\sigma$ is continuous, using \eqref{eq:sigma_alternative}, we have
		\begin{align*}
			&\lim_{\xi \rightarrow 0} \sigma(\xi) \\
			\overset{\eqref{eq:sigma_alternative}}&{=} \max_{\norm{\v{\alpha}}_p\leq 1} \min_{\v{\omega}} \max_i \lim_{\xi \rightarrow 0} \left|2\v{e}_i^\top\mat{H}^+\mat{G}_{\frac{1}{\xi}\mat{A}}\v{\alpha} + \v{e}_i^\top(I - \mat{H}^+\mat{H})\v{\omega}\right|\\
			\overset{\v{\alpha}=-\v{e}_{k}}&{\geq}  \min_{\v{\omega}} \max_i \lim_{\xi \rightarrow 0} \left|\frac{2}{\xi}L\v{e}_i^\top\mat{H}^+\begin{bmatrix}\v{a}_{k}\\0\end{bmatrix} + \v{e}_i^\top(I - \mat{H}^+\mat{H})\v{\omega}\right|\\
			&=\infty
		\end{align*}
		where $\v{a}_k$ denotes the $k$-th column of $\mat{A}$, with $k$ some index such that $\v{a}_k\neq \v{0}$.
		By the equivalence of norms on $\reals^m$, $\norm{\v{x}}_1 \leq m^{1-1/p}\norm{\v{x}}_p$ for any $\v{x}\in \reals^{m}$ and $p\in[1,\infty]$ (where we interpret $1/\infty$ as 0). This yields
		\begin{equation}
			\label{eq:norm_equivalence_A_norm}
			\begin{aligned}
				\norm{\mat{A}}_{p\mapsto 1} 
				&= \max_{\v{\alpha}\neq \v{0}} \frac{\norm{\mat{A}\v{\alpha}}_1}{\norm{\v{\alpha}}_p} && \text{(see \cite[Proposition III.2.1.]{conway_course_2007})}\\
				&\leq m^{\frac{p-1}{p}} \frac{\norm{\mat{A}\v{\alpha}}_1}{\norm{\v{\alpha}}_1} && \text{(by the equivalence of norms)}\\
				&= m^{\frac{p-1}{p}}\norm{\mat{A}}_{1\mapsto 1} && \text{(see \cite[Proposition III.2.1.]{conway_course_2007})}
			\end{aligned}
		\end{equation}
		Consequently, $\norm{\mat{A}}_{p\mapsto 1} / \widehat{\xi} \leq 1$, and so $\sigma(\widehat{\xi}) \leq 1$ follows by Lemma \ref{lmm:main_technical_lemma_apx}. Additionally, Lemma \ref{lmm:main_technical_lemma_apx} implies that $\xi = \norm{\mat{A}}_{p\mapsto 1}$ is the only solution to $\sigma(\xi) = 1$. We conclude by proving \eqref{eq:bound_sigma}: Let $\v{a}_1,...,\v{a}_m$ be the columns of $\mat{A}$, and assume $j\in\{1,...,m\}$ satisfies $\norm{\v{a}_j}_{\infty} = \max_{i}\norm{\v{a}_i}_{\infty} = \norm{\mat{A}}_{1\mapsto \infty}$. We can bound $\sigma(\xi)$ as
		\begin{align*}
			\sigma(\xi) 
			&= \max_{\norm{\v{\alpha}}_p\leq 1} \min_{\frac{1}{2}\mat{H}\v{\beta} = \mat{G}_{\frac{1}{\xi}\mat{A}}\v{\alpha}} \nnorm{\v{\beta}}_{\infty} && \text{(using \eqref{eq:sigma_def} and \eqref{eq:ellipsotope_containment})}\\
			&\geq \min_{\frac{1}{2}\mat{H}\v{\beta} = \frac{1}{\xi}\begin{bsmallmatrix}\v{a}_j\\0\end{bsmallmatrix}} \nnorm{\v{\beta}}_{\infty} & \begin{split}&\text{(definition of $\max$,}\\
				&\text{replacing $\v{\alpha}$ by $\v{e}_j$)}\end{split}\\
			&= \frac{1}{\xi}\norm{\v{a}_j}_{Z(\frac{1}{2}\mat{H})} && \text{(using Proposition \ref{prop:ellipsotope_norm})}
		\end{align*}
		Choosing $\v{\beta}$ as $\beta_i = a_{j,i}$ and $\beta_{i+n} = -a_{j,i}$ for $i=1,...,n$ shows $\norm{\v{a}_j}_{Z(\frac{1}{2}\mat{H})} \leq \norm{\v{a}_j}_{\infty}$. By duality,
		\begin{equation*}
			\norm{\v{a}_j}_{Z(\frac{1}{2}\mat{H})} = \max_{\norm{\frac{1}{2}\mat{H}^\top \v{y}}_1\leq 1} \v{y}^\top\v{a}_j.
		\end{equation*}
		Let $k$ be an index such that $|a_{j,k}| = \norm{\v{a}_j}_{\infty}$, then choosing $\v{y} = \v{e}_k$ or $\v{y} = -\v{e}_k$ shows $\norm{\v{a}_j}_{Z(\frac{1}{2}\mat{H})} \geq |a_{j,k}|$, and thus
		\begin{equation}
			\label{eq:sigma_lower_bound}
			\sigma(\xi) \geq \frac{\norm{\mat{A}}_{1\mapsto \infty}}{\xi}.
		\end{equation}
		For $\xi, \xi' \in (0, \widehat{\xi}]$,
		\begin{equation}
			\label{eq:inv_lower_bound}
			\left(\frac{1}{\xi} - \frac{1}{\xi'}\right) \geq \frac{1}{\widehat{\xi}^2}\left|\xi - \xi'\right|.
		\end{equation}
		Combining \eqref{eq:sigma_lower_bound} with \eqref{eq:inv_lower_bound} yields \eqref{eq:bound_sigma}.
	\end{proof}
	
	We now have the necessary tools to prove that the ellipsotope-in-zonotope containment problem is not in $\FPTAS$, unless $\Poly = \NP$:
	
	\textit{Proof of Theorem \ref{thm:not_FPTAS}:}
	As mentioned earlier, the key idea is to construct a bisection-type algorithm to find a solution to $\sigma(\xi) = 1$, using an approximation $\algo{approx}(\xi)$. This will give us a solution $\xi_{\star}$ which is close to $\norm{\mat{A}}_{p\mapsto 1}$, thus showing that one can approximate the $p\mapsto 1$-norm. The main challenge is to examine how close we can make $\xi_{\star}$ to $\norm{\mat{A}}_{p\mapsto 1}$.
	
	For a matrix $\mat{A}\in \reals^{n\times m}$ and $p\in(1,\infty)$, let $\sigma$ and $\widehat{\xi}$ be as defined in Proposition \ref{prop:apx_main_prop}, and assume that $\varrho$ and $L$ have been computed within floating-point error. Let $\delta > 0$ be a fixed accuracy parameter, and assume that the ellipsotope-in-zonotope containment problem is in $\FPTAS$, so that for any $\varepsilon > 0$, and $\xi \in (0, \widehat{\xi}\,)$ there exists an approximation algorithm $\algo{approx}(\xi)$ satisfying
	\begin{equation}
		\sigma(\xi) \leq \algo{approx}(\xi) \leq (1+\varepsilon) \sigma(\xi),
	\end{equation}
	running in polynomial time with respect to the representation size of $\mat{A}$ and $1/\varepsilon$. We set up a bisection-type algorithm applied on the initial interval $I = (0, \widehat{\xi}\,)$ for some small parameter $\mu > 0$:
	\begin{enumerate}[Step 1:]
		\item Let $a$ and $b$ be the lower and upper bounds, respectively, of $I$ (i.e., $I = (a,b)$).
		\item If $b - a \leq 2\mu$, let $\xi_{\star} \gets \frac{1}{2}(a+b)$ and return $\xi_{\star}$. Otherwise, continue with Step 3.
		\item Let $\xi_{\star} \gets \frac{1}{2}(a+b)$ be the center of $I$.
		\item If $\algo{approx}(\xi_{\star}) \leq 1$, set $I \gets (a, \xi_{\star})$ and go to Step~1.
		\item If $\algo{approx}(\xi_{\star}) \geq 1+\varepsilon$, set $I \gets (\xi_{\star}, b)$ and go to Step~1.
		\item If $1 < \algo{approx}(\xi_{\star}) < 1+\varepsilon$, return $\xi_{\star}$.
	\end{enumerate}
	Note that, unlike a traditional bisection algorithm, we do not directly search for a root of $\sigma-1$, nor do we try to find a root of $\algo{approx}-1$, as $\algo{approx}$ may not be continuous. Instead, we try to approximate a root of $\sigma-1$ using the approximation $\algo{approx}-1$, which requires a more detailed case-by-case analysis.
	
	Specifically, note that the algorithm terminates either in Step~2 or Step~6. In the first case, let $a = \xi_{\star} - \mu$ and $b = \xi_{\star} + \mu$. Since the overall structure of the algorithm is identical to a traditional bisection algorithm, by construction $\algo{approx}_{\varepsilon}(a) \geq 1 + \varepsilon$ and $\algo{approx}_{\varepsilon}(b) \leq 1$, thus
	\begin{equation*}
		\sigma(a) \geq \frac{1}{1+\varepsilon}\algo{approx}(a) \geq 1 \geq \algo{approx}(b) \geq \sigma(b).
	\end{equation*}
	Since by Proposition \ref{prop:apx_main_prop}, $\sigma$ is continuous and only $\xi = \norm{\mat{A}}_{p\mapsto 1}$ satisfies $\sigma(\xi) = 1$, by the intermediate value theorem
	\begin{equation}
		\label{eq:first_bounds_mu}
		\norm{\mat{A}}_{p\mapsto 1} \in [a, b] = \xi_{\star} + \mu[-1,1].
	\end{equation}
	Similarly as in \eqref{eq:norm_equivalence_A_norm}, one can prove $\norm{\mat{A}}_{1\mapsto 1} \leq \norm{\mat{A}}_{p\mapsto 1}$. Therefore, by choosing $\mu = \frac{\delta}{2}\norm{\mat{A}}_{1\mapsto 1}$, we obtain the relation
	\begin{equation}
		\label{eq:final_step_2_result}
		\norm{\mat{A}}_{p\mapsto 1} \leq \xi_{\star} + \frac{\delta}{2}\norm{\mat{A}}_{1\mapsto 1} \leq (1+\delta)\norm{\mat{A}}_{p\mapsto 1}.
	\end{equation}
	On the other hand, if the algorithm terminates in Step 6, $\xi_{\star}$ satisfies
	\begin{align*}
		&\sigma(\xi_{\star}) \leq \algo{approx}(\xi_{\star}) \leq 1+\varepsilon,\\
		&\sigma(\xi_{\star}) \geq \frac{1}{1+\varepsilon}\algo{approx}(\xi_{\star}) \geq \frac{1}{1+\varepsilon}.
	\end{align*}
	By using \eqref{eq:bound_sigma} on $\xi = \norm{\mat{A}}_{p\mapsto 1}$ (in which case $\sigma(\xi) = 1$) and $\xi' = \xi_{\star}$, we obtain
	\begin{equation*}
		\frac{\norm{\mat{A}}_{1\mapsto \infty}}{\widehat{\xi}^2}\left| \norm{\mat{A}}_{p\mapsto 1} - \xi_{\star}\right| \leq  \left|1 - \sigma(\xi_{\star})\right| \in \left[\frac{\varepsilon}{1+\varepsilon}, \varepsilon\right]
	\end{equation*}
	which implies
	\begin{equation*}
		\norm{\mat{A}}_{p\mapsto 1} \leq \xi_{\star} + \frac{\widehat{\xi}^2\varepsilon}{\norm{\mat{A}}_{1\mapsto \infty}} \leq \norm{\mat{A}}_{p\mapsto 1} + 2\frac{\widehat{\xi}^2\varepsilon}{\norm{\mat{A}}_{1\mapsto \infty}}.
	\end{equation*}
	Using a similar argument as for \eqref{eq:final_step_2_result}, we may choose
	\begin{equation}
		\label{eq:choice_epsilon}
		\varepsilon = \frac{\delta\norm{\mat{A}}_{1\mapsto 1}\norm{\mat{A}}_{1\mapsto \infty}}{2\widehat{\xi}^2},
	\end{equation}
	which yields
	\begin{equation*}
		\norm{\mat{A}}_{p\mapsto 1} \leq \xi_{\star} + \frac{\delta\norm{\mat{A}}_{1\mapsto 1}}{2} \leq (1+\delta)\norm{\mat{A}}_{p\mapsto 1}.
	\end{equation*}
	Therefore, we always end up with a $(1+\delta)$-approximation for $\norm{\mat{A}}_{p\mapsto 1}$. Concerning the runtime, let $N_{\mathrm{bisect}}$ be the number of times the interval $I$ has to be halved in the bisection-type algorithm above. Just like for any bisection algorithm,
	\begin{equation*}
		N_{\mathrm{bisect}} \leq \log_2(\widehat{\xi}/\mu) - 1 = \frac{p-1}{p}\log_2(m/\delta).
	\end{equation*}
	Since $\log_2(x) \leq x$ for any $x\in\reals$, $N$ can be bounded by a polynomial in $m$ and $1/\delta$. On the other hand, since $\norm{\v{x}}_1 \leq n\norm{\v{x}}_{\infty}$ for any $\v{x} \in \reals^n$ implies $\norm{\mat{A}}_{1\mapsto 1}\leq n\norm{\mat{A}}_{1\mapsto \infty}$, our choice of $\varepsilon$ in \eqref{eq:choice_epsilon} yields
	\begin{align*}
		\frac{1}{\varepsilon}  &= \frac{2\widehat{\xi}^2}{\delta\norm{\mat{A}}_{1\mapsto 1}\norm{\mat{A}}_{1\mapsto \infty}} && \text{(using \eqref{eq:choice_epsilon})}\\
		&\leq \frac{2m^{\frac{2p-2}{p}}\norm{\mat{A}}_{1\mapsto 1}^2}{\delta\norm{\mat{A}}_{1\mapsto 1}\norm{\mat{A}}_{1\mapsto \infty}} && \text{(using \eqref{eq:widehat_xi_def})}\\
		&\leq \frac{2nm^{\frac{2p-2}{p}}\norm{\mat{A}}_{1\mapsto 1}^2}{\delta\norm{\mat{A}}_{1\mapsto 1}^2} && \text{(using $\norm{\mat{A}}_{1\mapsto 1}\leq n\norm{\mat{A}}_{1\mapsto \infty}$)}\\
		&= \frac{2nm^{\frac{2p-2}{p}}}{\delta},
	\end{align*}
	which is polynomial in $1/\delta$, $n$, and $m$. Finally, note that $\norm{\mat{A}}_{1\mapsto 1} = \max_i \sum_j |A_{ji}|$ and $\norm{\mat{A}}_{1\mapsto \infty} = \max_{i,j} |A_{ij}|$ can be computed in polynomial time with respect to $n$ and $m$. Overall, this means that executing our modified bisection algorithm requires at most $N_{\mathrm{bisect}}$ evaluations of the approximation algorithm $\algo{approx}$, evaluated with an accuracy $\varepsilon$ such that $1/\varepsilon$ is polynomial in $n$, $m$, $l$, and $1/\delta$, so that the entire algorithm has polynomial runtime. Consequently, we found a $(1+\delta)$-approximation algorithm for $\norm{\mat{A}}_{p\mapsto 1}$, meaning that computing $\norm{\mat{A}}_{p\mapsto 1}$ would be in $\FPTAS$, which is impossible according to \cite[Theorem 6.4.]{bhaskara} for $1<p<2$ and \cite[Theorem 1.4.]{bhattiprolu_2023} for $2\leq p \leq \infty$, unless $\Poly = \NP$. Moreover, since by Proposition \ref{prop:apx_main_prop} it holds that $\norm{\mat{A}}_{p\mapsto 1} > 1$ if and only if $r(\Ein, \Zcirc) > 1$, the decision formulation of the containment problem is $\coNP$-hard. So far, we assumed $p < \infty$, but for $p=\infty$ the arguments above still hold by replacing any occurrence of $1/p$ and $p/p$ by $0$ and $1$, respectively. This completes the proof of the Theorem.
	
	$\hfill\blacksquare$
	
	\subsection{Proof of Theorem \ref{thm:RP_RQP_zono_containment}}
	\label{sec:proof_thm_RP_RQP}
	Let $p \in (2, \infty)$ be fixed, $\mat{A}\in\reals^{m\times n}$, and $\mat{H} \in \reals^{n\times l}$ be a matrix we shall choose later. Since we can always complete $\mat{A}$ with zeros without impacting the value of $\norm{\mat{A}}_{p^*\mapsto p^*}$, we may assume $m=n=\widetilde{N}$. We assume that $r(E_p(\mat{A}^\top), Z(\mat{H}))$ can be approximated in polynomial time (with respect to the representation sizes of $\mat{A}$ and $\mat{H}$) by some algorithm $\algo{approx}(\mat{A}, \mat{H})$ which achieves an approximation ratio $\tau \geq 1$, and let $\varepsilon > 0$ be some small parameter (for instance $\varepsilon \leq 1/3$). By \eqref{eq:ellipsotope_containment_equiv},
	\begin{equation*}
		r(E_p(\mat{A}^\top), Z(\mat{H})) = \max_{\norm{\mat{H}^\top \v{y}}_1 \leq 1} \norm{\mat{A} \v{y}}_{p^*}.
	\end{equation*}
	According to \cite[Theorem 4.12]{bhattiprolu_2023}, there exists a probability distribution on $\reals^{l\times n}$ with $l$ growing at most polynomially in $n$ (though it may grow exponentially with respect to $1/\varepsilon$) that can produce a matrix $\mat{B} \in \reals^{l\times n}$ such that
	\begin{equation}
		\label{eq:mat_B_property}
		(1-\varepsilon)\norm{\v{y}}_{p^*} \leq \norm{\mat{B}\v{y}}_1 \leq (1+\varepsilon)\norm{\v{y}}_{p^*}
	\end{equation}
	holds with probability $P$, where $P \in (0,1)$ is a constant that can depend on $p$ but not on $n$ nor $\varepsilon$. Consequently,
	\begin{align*}
		&\norm{\mat{A}}_{p^*\mapsto p^*}\\
		&=\max_{\v{x}\neq\v{0}} \frac{\norm{\mat{A}\v{x}}_{p^*}}{\norm{\v{x}}_{p^*}} && \text{(see \cite[Proposition III.2.1]{conway_course_2007})}\\
		&\leq \max_{\v{x}\neq\v{0}} (1+\varepsilon)\frac{\norm{\mat{A}\v{x}}_{p^*}}{\norm{\mat{B}\v{x}}_{1}} && \text{(using \eqref{eq:mat_B_property})}\\
		&=(1+\varepsilon)\max_{\norm{\mat{B}\v{x}}_{1}\leq 1} \norm{\mat{A}\v{x}}_{p^*} && \text{(see \cite[Proposition III.2.1]{conway_course_2007})}\\
		&\leq (1+\varepsilon)\algo{approx}(\mat{A}, \mat{B}^\top) && \text{(definition of $\algo{approx}$)}\\
		&\leq (1+\varepsilon)\tau \cdot \max_{\norm{\mat{B}\v{x}}_{1}\leq 1} \norm{\mat{A}\v{x}}_{p^*} && \text{(definition of $\algo{approx}$)}\\
		&\leq \frac{1+\varepsilon}{1-\varepsilon}\tau\norm{\mat{A}}_{p^*\mapsto p^*}. && \text{(using \eqref{eq:mat_B_property})}
	\end{align*}
	We can thus conclude
	\begin{equation}
		\label{eq:prop_bound_A_p_dual_p_dual}
		\norm{\mat{A}}_{p^*\mapsto p^*} \leq (1+\varepsilon)\algo{approx}(\mat{A},\mat{B}^\top) \leq \frac{1+\varepsilon}{1-\varepsilon}\tau\norm{\mat{A}}_{p^*\mapsto p^*}.
	\end{equation}
	If $\varepsilon \leq 1/3$, this yields a $2\tau$-approximation that runs in polynomial time and returns the correct result with a probability of at least $P$. Without loss of generality, we may assume $P \geq 1/2$, since it suffices to repeat the algorithm several times for different $\mat{B}$ to obtain, with a probability of at least $1/2$, that \eqref{eq:prop_bound_A_p_dual_p_dual} holds for at least some $\mat{B}$, and thus taking the maximum of $(1-\varepsilon)\algo{approx}(\mat{A}, \mat{B}^\top)$ over all such $\mat{B}$ yields a $2\tau$-approximation that holds with a probability of at least $1/2$. According to \cite[Theorem 6.2.]{bhaskara}, approximating $\norm{\mat{A}}_{s\mapsto s}$ for $s \in (2,\infty)$, and thus for $s\in (1,2)$ by duality, within any constant ratio is $\NP$-hard. Therefore, the existence of an algorithm as described above would entail $\NP \subseteq \RP$. Furthermore, \cite[Theorem 6.2.]{bhaskara} proved that $\norm{\mat{A}}_{s\mapsto s}$ is inapproximable within a factor $\Omega(2^{(\log \widetilde{N})^{1-\delta}})$ for any $\delta > 0$ unless every problem in $\NP$ can be solved by an algorithm that runs in quasi-polynomial time with respect to $N$. Consequently, the ellipsotope-in-zonotope containment problem must be inapproximable within a factor $\Omega(2^{(\log N)^{1-\delta}})$, unless $\NP \subseteq \RQP$.
	
	For $p = \infty$, we instead consider approximations $\algo{approx}(\mat{A}, \mat{B})$ of $r(Z(\mat{A}^\top\mat{B}^\top), Z(\mat{B}^\top))$, where $\mat{B}$ is a random matrix which satisfies \eqref{eq:mat_B_property} with probability $P$, with $p^*$ replaced by some $q\in(1,2)$. Similarly to the case $p\in(2,\infty)$ we can deduce
	\begin{equation}
		\label{eq:prop_p=inf_replacement}
		\norm{\mat{A}}_{q\mapsto q} \leq \frac{1+\varepsilon}{1-\varepsilon}\algo{approx}(\mat{A}, \mat{B}) \leq \frac{(1+\varepsilon)^2}{(1-\varepsilon)^2}\tau\norm{\mat{A}}_{q\mapsto q}.
	\end{equation}
	The rest of the proof is identical to the $p\in(2,\infty)$ case, with the sole exception being the substitution of \eqref{eq:prop_bound_A_p_dual_p_dual} by \eqref{eq:prop_p=inf_replacement}, yielding a $4\tau$-approximation if $\varepsilon \leq 1/3$ (instead of a $2\tau$-approximation).
	
	$\hfill\blacksquare$

	\section*{Acknowledgment}
	
	We thank Vijay Bhattiprolu for his pivotal suggestions that led to the proof of Theorem \ref{thm:RP_RQP_zono_containment}. We would also like to thank Mark Wetzlinger and Matthias Mayer for proofreading this manuscript.
	
	\section*{REFERENCES}
	
	\bibliographystyle{IEEEtran}
	\bibliography{references}

\begin{thebibliography}{10}
\providecommand{\url}[1]{#1}
\csname url@samestyle\endcsname
\providecommand{\newblock}{\relax}
\providecommand{\bibinfo}[2]{#2}
\providecommand{\BIBentrySTDinterwordspacing}{\spaceskip=0pt\relax}
\providecommand{\BIBentryALTinterwordstretchfactor}{4}
\providecommand{\BIBentryALTinterwordspacing}{\spaceskip=\fontdimen2\font plus
\BIBentryALTinterwordstretchfactor\fontdimen3\font minus
  \fontdimen4\font\relax}
\providecommand{\BIBforeignlanguage}[2]{{%
\expandafter\ifx\csname l@#1\endcsname\relax
\typeout{** WARNING: IEEEtran.bst: No hyphenation pattern has been}%
\typeout{** loaded for the language `#1'. Using the pattern for}%
\typeout{** the default language instead.}%
\else
\language=\csname l@#1\endcsname
\fi
#2}}
\providecommand{\BIBdecl}{\relax}
\BIBdecl

\bibitem{althoff_reachability_2013}
M.~Althoff, ``Reachability analysis of nonlinear systems using conservative
  polynomialization and non-convex sets,'' in \emph{Proceedings of the 16th
  international conference on {Hybrid} {S}ystems: {C}omputation and {C}ontrol},
  2013, pp. 173--182.

\bibitem{rakovic_optimized_2007}
S.~V. Raković, E.~C. Kerrigan, D.~Q. Mayne, and K.~I. Kouramas, ``Optimized
  robust control invariance for linear discrete-time systems: Theoretical
  foundations,'' \emph{Automatica}, vol.~43, pp. 831--841, 2007.

\bibitem{ren_zonotope-based_2021}
W.~Ren, J.~Calbert, and R.~Jungers, ``Zonotope-based controller synthesis for
  {LTL} specifications,'' in \emph{60th {IEEE} {Conference} on {Decision} and
  {Control}}, 2021, pp. 580--585.

\bibitem{Roehm2016}
H.~Roehm, J.~Oehlerking, M.~Woehrle, and M.~Althoff, ``Reachset conformance
  testing of hybrid automata,'' in \emph{Proceedings of Hybrid Systems:
  Computation and Control}, 2016, pp. 277--286.

\bibitem{kellner_containment_2013}
K.~Kellner, T.~Theobald, and C.~Trabandt, ``Containment problems for polytopes
  and spectrahedra,'' \emph{SIAM Journal on Optimization}, vol.~23, pp.
  1000--1020, 2013.

\bibitem{gritzmann_complexity_1994}
P.~Gritzmann and V.~Klee, ``On the complexity of some basic problems in
  computational convexity: {I}. {Containment} problems,'' \emph{Discrete
  Mathematics}, vol. 136, pp. 129--174, 1994.

\bibitem{Kulmburg2021}
A.~Kulmburg and M.~Althoff, ``On the {co-NP}-completeness of the zonotope
  containment problem,'' \emph{European Journal of Control}, vol.~62, pp.
  84--91, 2021.

\bibitem{sadraddini_linear_2019}
S.~Sadraddini and R.~Tedrake, ``Linear encodings for polytope containment
  problems,'' in \emph{IEEE 58th Conference on Decision and Control}, 2019, pp.
  4367--4372.

\bibitem{ellipsotopes}
S.~Kousik, A.~Dai, and G.~X. Gao, ``Ellipsotopes: Uniting ellipsoids and
  zonotopes for reachability analysis and fault detection,'' \emph{IEEE
  Transactions on Automatic Control}, pp. 1--13, 2022.

\bibitem{gruber}
F.~Gruber and M.~Althoff, ``Computing safe sets of linear sampled-data
  systems,'' \emph{IEEE Control Systems Letters}, vol.~5, no.~2, pp. 385--390,
  2021.

\bibitem{horn2012}
R.~A. Horn and C.~R. Johnson, \emph{{Matrix analysis}}, 2nd~ed.\hskip 1em plus
  0.5em minus 0.4em\relax Cambridge University Press, 2012.

\bibitem{buehler2018}
T.~B{\"u}hler and D.~Salamon, \emph{Functional Analysis}.\hskip 1em plus 0.5em
  minus 0.4em\relax American Mathematical Society, 2018.

\bibitem{boyd_convex_2004}
S.~P. Boyd and L.~Vandenberghe, \emph{{Convex optimization}}.\hskip 1em plus
  0.5em minus 0.4em\relax Cambridge University Press, 2004.

\bibitem{ziegler2012lectures}
G.~Ziegler, \emph{Lectures on Polytopes}, 1st~ed.\hskip 1em plus 0.5em minus
  0.4em\relax Springer New York, 2012, updated Seventh Printing.

\bibitem{victor_ellipsoid}
V.~Gaßmann and M.~Althoff, ``Scalable zonotope-ellipsoid conversions using the
  {E}uclidean zonotope norm,'' in \emph{American Control Conference}, 2020, pp.
  4715--4721.

\bibitem{ellipsoidal_toolbox}
A.~A. Kurzhanskiy and P.~Varaiya, ``Ellipsoidal toolbox,'' Electrical
  Engineering and Computer Sciences University of California at Berkeley, Tech.
  Rep. UCB/EECS-2006-46, 2006.

\bibitem{narici2010topological}
L.~Narici and E.~Beckenstein, \emph{Topological Vector Spaces}.\hskip 1em plus
  0.5em minus 0.4em\relax CRC Press, 2010.

\bibitem{vazirani2013approximation}
V.~Vazirani, \emph{Approximation Algorithms}.\hskip 1em plus 0.5em minus
  0.4em\relax Springer Berlin Heidelberg, 2013.

\bibitem{Arora_Barak_2009}
S.~Arora and B.~Barak, \emph{Computational Complexity: A Modern
  Approach}.\hskip 1em plus 0.5em minus 0.4em\relax Cambridge University Press,
  2009.

\bibitem{ausiello_complexity_1999}
G.~Ausiello, M.~Protasi, A.~Marchetti-Spaccamela, G.~Gambosi, P.~Crescenzi, and
  V.~Kann, \emph{Complexity and Approximation: Combinatorial Optimization
  Problems and Their Approximability Properties}.\hskip 1em plus 0.5em minus
  0.4em\relax Springer-Verlag, 1999.

\bibitem{roman_advanced_2008}
S.~Roman, \emph{Advanced Linear Algebra}.\hskip 1em plus 0.5em minus
  0.4em\relax Springer, 2008, vol. 135.

\bibitem{kulmburg2023generalized}
A.~Kulmburg, ``The generalized matrix norm problem,'' 2025, to appear in the
  SIAM Journal on Matrix Analysis and Applications (SIMAX).

\bibitem{ellipsoids2022}
V.~Gaßmann and M.~Althoff, ``Implementation of ellipsoidal operations in
  {CORA} 2022,'' in \emph{Proceedings of 9th International Workshop on Applied
  Verification of Continuous and Hybrid Systems}, vol.~90.\hskip 1em plus 0.5em
  minus 0.4em\relax EasyChair, 2022, pp. 1--17.

\bibitem{james_generalised_1978}
M.~James, ``The generalised inverse,'' \emph{The Mathematical Gazette},
  vol.~62, no. 420, pp. 109--114, 1978.

\bibitem{planitz_3_1979}
M.~Planitz, ``{Inconsistent} systems of linear equations,'' \emph{The
  Mathematical Gazette}, vol.~63, no. 425, pp. 181--185, 1979.

\bibitem{kulmburgSearchbasedStochasticSolutions2024}
A.~Kulmburg, I.~Brkan, and M.~Althoff, ``Search-based and stochastic solutions
  to the zonotope and ellipsotope containment problems,'' in \emph{2024
  European Control Conference}, 2024, pp. 1057--1064.

\bibitem{tomczak-jaegermann_banach-mazur_1989}
N.~Tomczak-Jaegermann, \emph{Banach-{Mazur} Distances and Finite-dimensional
  Operator Ideals}.\hskip 1em plus 0.5em minus 0.4em\relax Longman Scientific
  \& Technical, 1989.

\bibitem{bhattiprolu_2023}
V.~Bhattiprolu, M.~K. Ghosh, V.~Guruswami, E.~Lee, and M.~Tulsiani,
  ``Inapproximability of matrix {\(\boldsymbol{p \rightarrow q}\)} norms,''
  \emph{SIAM Journal on Computing}, vol.~52, no.~1, pp. 132--155, 2023.

\bibitem{bhaskara}
A.~Bhaskara and A.~Vijayaraghavan, ``Approximating matrix p-norms,'' in
  \emph{Proceedings of the Annual ACM-SIAM Symposium on Discrete Algorithms},
  2011, p. 497–511.

\bibitem{barak_brandao_et_al}
B.~Barak, F.~G. Brandao, A.~W. Harrow, J.~Kelner, D.~Steurer, and Y.~Zhou,
  ``Hypercontractivity, sum-of-squares proofs, and their applications,'' in
  \emph{Proceedings of the Annual ACM Symposium on Theory of Computing}, 2012,
  pp. 307--326.

\bibitem{burger_finding_2000}
T.~Burger and P.~Gritzmann, ``Finding optimal shadows of polytopes,''
  \emph{Discrete \& Computational Geometry}, vol.~24, pp. 219--240, 2000.

\bibitem{althoff_dissertation}
M.~Althoff, ``Reachability analysis and its application to the safety
  assessment of autonomous cars,'' Ph.D. dissertation, Technische Universität
  München, 2010.

\bibitem{Rakovic2017_sampledData}
S.~Raković, F.~Fontes, and I.~Kolmanovsky, ``Reachability and invariance for
  linear sampled–data systems,'' \emph{IFAC-PapersOnLine}, vol.~50, no.~1,
  pp. 3057--3062, 2017.

\bibitem{Lewis2012}
F.~Lewis, D.~Vrabie, and V.~Syrmos, \emph{Optimal Control}.\hskip 1em plus
  0.5em minus 0.4em\relax Wiley, 2012.

\bibitem{raghuraman_set_2022}
V.~Raghuraman and J.~P. Koeln, ``Set operations and order reductions for
  constrained zonotopes,'' \emph{Automatica}, vol. 139, 2022.

\bibitem{schaefer}
L.~Schäfer, F.~Gruber, and M.~Althoff, ``Scalable computation of robust
  control invariant sets of nonlinear systems,'' \emph{IEEE Transactions on
  Automatic Control}, vol.~69, no.~2, pp. 755--770, 2024.

\bibitem{schuermann2017}
B.~Schürmann and M.~Althoff, ``Optimal control of sets of solutions to
  formally guarantee constraints of disturbed linear systems,'' in
  \emph{American Control Conference}, 2017, pp. 2522--2529.

\bibitem{cora}
M.~Althoff, ``An introduction to {CORA} 2015,'' in \emph{Proceedings of the
  Workshop on Applied Verification for Continuous and Hybrid Systems}, 2015,
  pp. 120 -- 151.

\bibitem{mosek}
\BIBentryALTinterwordspacing
M.~ApS, \emph{{The MOSEK optimization toolbox for MATLAB manual. Version
  10.0.}}, 2022. [Online]. Available:
  \url{http://docs.mosek.com/10.0/toolbox/index.html}
\BIBentrySTDinterwordspacing

\bibitem{Lofberg2004}
J.~L{\"{o}}fberg, ``{YALMIP} : A toolbox for modeling and optimization in
  {MATLAB},'' in \emph{Proceedings of the CACSD Conference}, 2004.

\bibitem{aroc}
N.~Kochdumper, F.~Gruber, B.~Schürmann, V.~Gaßmann, M.~Klischat, and
  M.~Althoff, ``{AROC}: A toolbox for automated reachset optimal controller
  synthesis,'' in \emph{Proceedings of the 24th International Conference on
  Hybrid Systems: Computation and Control}, 2021.

\bibitem{ALTHOFF2010233}
M.~Althoff, O.~Stursberg, and M.~Buss, ``Computing reachable sets of hybrid
  systems using a combination of zonotopes and polytopes,'' \emph{Nonlinear
  Analysis: Hybrid Systems}, vol.~4, no.~2, pp. 233--249, 2010.

\bibitem{conway_course_2007}
J.~B. Conway, \emph{A Course in Functional Analysis}.\hskip 1em plus 0.5em
  minus 0.4em\relax Springer, 2007.

\end{thebibliography}
	
	\begin{IEEEbiography}[{\includegraphics[width=1in,height=1.25in,clip,keepaspectratio]{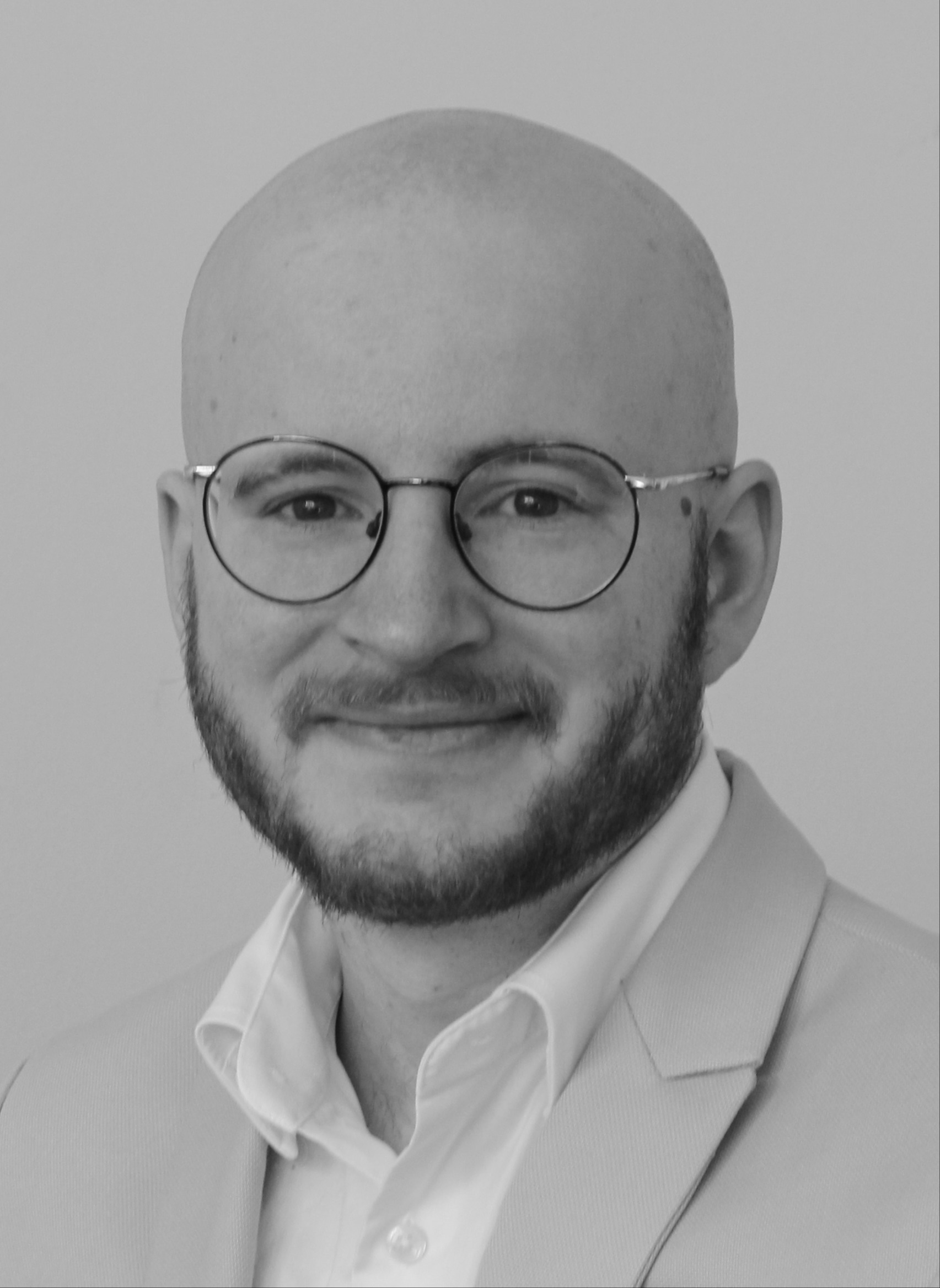}}]{Adrian Kulmburg} received the B.Sc. degree in mathematics in 2017, the M.Sc. degree in mathematics in 2018, and the B.Sc. degree in physics in 2019, all from the Federal Institute of Technology of Zürich, Zürich, Switzerland.
		
		He joined the Cyber-Physical Systems Group at the Technical University of Munich in 2020. His research interests focus on differential equations, discrete geometry, functional analysis, and computational complexity.
	\end{IEEEbiography}
	\begin{IEEEbiography}[{\includegraphics[width=1in,height=1.25in,clip,keepaspectratio]{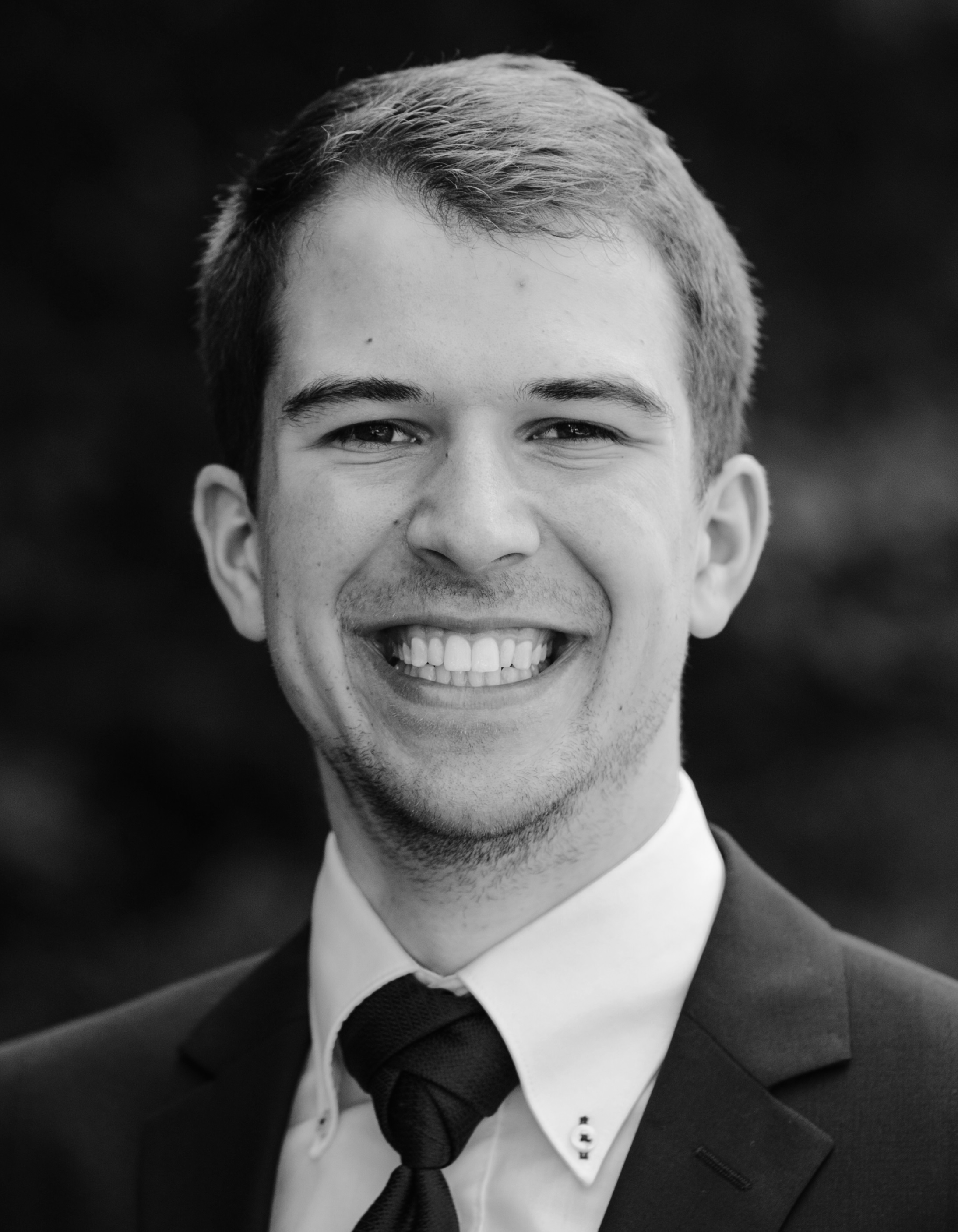}}]{Lukas Schäfer} received the B.Eng. degree in business administration and engineering from Baden-Wuerttemberg Cooperative State University, Stuttgart, Germany, in 2015, and the B.Sc. degree in mechanical engineering and the M.Sc. degree in robotics, cognition, intelligence both from the Technical University of Munich, Munich, Germany, in 2018 and 2021, respectively.
		
		He joined the Cyber–Physical Systems Group, Technical University of Munich, in 2021. His research focuses on robust nonlinear model predictive control and optimal control. Application-wise, he is interested in motion planning for autonomous vehicles (including vehicle dynamics modeling), as well as safeguarding reinforcement learning agents.
	\end{IEEEbiography}
	\begin{IEEEbiography}[{\includegraphics[width=1in,height=1.25in,clip,keepaspectratio]{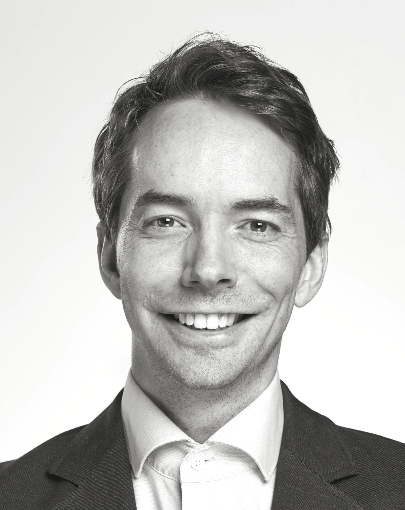}}]{Matthias Althoff} (Member, IEEE) received the diploma engineering degree in mechanical engineering and the Ph.D. degree in electrical engineering both from the Technical University of Munich, Munich, Germany, in 2005 and 2010,	respectively.
		
		He is an Associate Professor in Computer Science at the Technical University of Munich. From 2010 to 2012, he was a Postdoctoral Researcher at Carnegie Mellon University, Pittsburgh, PA, USA, and from 2012 to 2013, he was an Assistant Professor at Ilmenau University of Technology, Ilmenau, Germany. His research interests include the formal verification of continuous and hybrid systems, reachability analysis, planning algorithms, nonlinear control, automated vehicles, and power systems.
	\end{IEEEbiography}
	
\end{document}